\documentclass[11pt]{article}
\usepackage{epsfig}
\usepackage{amssymb,amsmath,amsthm,amscd}
\usepackage{amsfonts, mathabx}
\usepackage{latexsym}
 \usepackage{bbm,dsfont}

 \usepackage[notref,notcite]{showkeys}

\newtheorem{thm}{Theorem}[section]
\newtheorem{prop}[thm]{Proposition}

\newtheorem{lem}[thm]{Lemma}

\newtheorem{defn}[thm]{Definition}

\newcommand{\be}{\begin{equation}}
\newcommand{\ee}{\end{equation}}

\newcommand{\Z}{\mathbb{Z}^N}
\newcommand{\R}{\mathbb{R}}
\newcommand{\N}{\mathbb{N}}
\newcommand{\E}{\mathbb{E}}

\def \eps {{ \varepsilon }}

\def \calf {{  {\mathcal{F}} }}

\def \cala {{  {\mathcal{A}}  }}

\def \cale {{  {\mathcal{E}}  }}
\def \calg {{  {\mathcal{G}}  }}

 \def \calt {{  {\mathcal{T}}  }}

\begin{document}

\begin{titlepage}
\title{\bf 
Large Deviation Principles
 of Stochastic Reaction-Diffusion   Lattice Systems
 }
\vspace{7mm}

\author{
 Bixiang Wang  
\vspace{1mm}\\
Department of Mathematics, New Mexico Institute of Mining and
Technology \vspace{1mm}\\ Socorro,  NM~87801, USA \vspace{1mm}\\
Email: bwang@nmt.edu\vspace{6mm}\\
  }

\date{}
\end{titlepage}

\maketitle

\medskip

\begin{abstract}
 This paper is concerned with the large deviation principle
 of the   stochastic reaction-diffusion  lattice systems 
 defined on the $N$-dimensional  integer set, where the nonlinear
 drift term is locally Lipschitz continuous with polynomial growth
 of any degree and the nonlinear diffusion term is locally
 Lipschitz continuous with linear growth.
 We first prove the convergence of the solutions of the
 controlled stochastic lattice  systems, and then establish the large
 deviations by the weak convergence method based on
 the equivalence of the  large deviation principle
  and the Laplace
 principle.
  \end{abstract}



\baselineskip=1.25\baselineskip

\section{Introduction} 
\setcounter{equation}{0}

In this paper, we investigate the
large deviation principle
of  the  non-autonomous stochastic
reaction-diffusion
 lattice system defined on the
 $N$-dimensional  integer
 set $\Z$.
 Given   $i\in   \Z$, consider the 
 It\^{o} stochastic system:
 \be\label{intr1}
 d u^\eps_i(t)
 + \nu (Au^\eps(t))_i
  dt 
 +F_i (t, u^\eps_i(t) ) dt  
= \sqrt{\eps}
 \sum_{k=1}^\infty 
  \sigma_{k,i}   (t, u^\eps_i (t))
  dW_k, \quad t>0,
\ee
 with  initial data
 \be\label{intr2}
 u^\eps_i(0)=u_{0,i},
 \ee
  where 
  $i=(i_1,\ldots, i_N)\in \Z$,
$u=(u_i)_{i\in \Z}$
is an unknown sequence,
  $\nu>0$ and $\eps\in (0,1)$  are 
constants, 
$A$ is the  negative discrete 
$N$-dimensional Laplace operator
defined on $\Z$, 
  and 
   $ (W_k)_{k\in \N}$  is a 
sequence of  independent   real-valued 
standard Wiener 
processes on a probability space.  

For every $i\in \Z$, the nonlinear
function
 $F_i:  \R^+\times \R \to \R$  
 is a  locally Lipschitz function with polynomial
 growth of any degree
 with respect to the second argument.
 For the diffusion coefficients,
 we will assume that
 for every $i\in \Z$  and
 $k\in \N$, 
$  \sigma_{k,i}: \R^+\times \R \to \R$ 
 is a locally 
 Lipschitz function with linear 
 growth 
 with respect to the second argument.

Lattice  systems  can be used to describe
the dynamics of physical systems with
discrete structures, including
electric circuits,
pattern formation
and propagation of nerve pulses
 \cite{bel2, chua1, chua2,
ern1, kap1, kee1, kee2}.
Such systems also occur by discretizing partial differential
equations in space variables defined on unbounded domains.
The solutions of 
both deterministic and stochastic lattice
systems have been investigated    by many experts.
For  deterministic lattice systems, the
traveling waves, chaotic solutions and global
attractors have been studied in 
 \cite{afr1, bates1, bates2,  cho2, elm1, elm2, ern1, zin1}, 
   \cite{cho1, cho3, cho4}
  and  
\cite{bates3,    bey1,  han3, kara1, mori1,  wan8},
respectively.
For stochastic lattice systems,
the random attractors
have been reported  in    \cite{bat1, bat3, 
  car7, car8, han1, han2}.
  Recently, the 
   invariant measures
   and periodic measures of
   stochastic lattice systems
   have 
   been examined in
  \cite{clw,chen1,chen2,
  lww1, lww2, wan2019, rwan1, xwan1}
  and the references therein.
  In the present paper, we study the
  large deviation principle of the
  stochastic lattice system
  \eqref{intr1}-\eqref{intr2}.
  
  The large deviation principle  of
  stochastic systems  
 is concerned with the exponential decay
 of distributions  of
 solutions
 on  tail events as $\eps \to 0$, which has 
    been
  investigated in
  \cite{dem1, fre2, str1,
  var1, var2} and the references therein.
There are two basic approaches to deal with   
the 
  large deviations 
  of stochastic partial differential equations:
  the classical method and the weak convergence method.
 The classical method is based on 
  the 
   discretization and approximation arguments
   along with     uniform exponential probability estimates, 
 see, e.g.,  \cite{bra1, card1, cerr1, che1, chow1, 
  far1, fre1, gau1, kall1, mart1, sow1}.
  The weak convergence method 
   is based on the equivalence of
  large deviation principles and Laplace principles
  as well as variational representations of
  positive functions of infinite-dimensional Brownian motions,
  see, e.g., 
   \cite{bess1, brz1, bud1, bud2,
  cerr2, cerr3, chu1, duan1, dup1, liu1, ort1, roc1,sal1}.
  
The   large deviation principle
has  been well developed  for finite-dimensional
 dynamical
systems generated by stochastic ordinary
differential equations and infinite-dimensional dynamical
systems generated by stochastic partial differential equations.
However,
as far as the author is aware, it seems that
  there is not a result available in the literature
   on large deviations for
  infinite-dimensional lattice systems.
  The goal of the present
   paper is to investigate this problem
  and
  establish the  
  large deviation principle  of the infinite-dimensional
  lattice system \eqref{intr1}-\eqref{intr2} by employing
  the weak convergence method
  as introduced in \cite{bud1, bud2, dup1}. 
   One of the advantages of the weak convergence method
  lies in the fact that it does not require the
  uniform exponential probability estimates of solutions.

  Note that the lattice system \eqref{intr1}
  on $\Z$
  can be considered as spatial discretization
  of the corresponding  reaction-diffusion partial differential equation
  on $\R^n$.
  The large deviations
  of reaction-diffusion equations in bounded domains 
  have been studied
  in \cite{cerr1, cerr2, hai1}
  when the nonlinear drift term is locally Lipschitz
  continuous
  with polynomial growth and the diffusion term is globally
  Lipschitz continuous.
 In the present paper, we  will
 deal with  the infinite-dimensional
  lattice system \eqref{intr1}
  with polynomial drift term and  locally Lipschitz 
  diffusion term.
  The reader is referred to \cite{cerr1a, che1, fre1, pes1, sow1}
  for large deviations of reaction-diffusion equations
  in bounded domains with globally Lipschitz drift terms.

   We will recall basic concepts of large deviation
   principles and Laplace principles in the next section,
   and discuss the well-posedness of system
   \eqref{intr1}-\eqref{intr2} in Section 3. 
   We finally prove the large deviation principle of
   \eqref{intr1}-\eqref{intr2} in Section 4.

 \section{Large deviation theory}
 
 In this section, we review the large deviation
 principle and the Laplace principle of a family
 of random variables
 based on the weak convergence method
 as  introduced
 in \cite{bud1, dup1}.

   Let $(\Omega, \mathcal{F}, 
   \{ { \mathcal{F}} _t\} _{t\ge 0},  P )$
be a  complete filtered probability 
space satisfying the usual condition.
Suppose $\{W(t)\}_{t\ge 0} $ is 
a cylindrical Wiener process
with identity covariance operator 
in a separable Hilbert space $H$
with respect to 
$(\Omega, \mathcal{F}, 
   \{ { \mathcal{F}} _t\} _{t\ge 0},  P )$,
   which means
   that there exists another separable 
   Hilbert space $U$
   such that
    the embedding
  $H\hookrightarrow U$ is   Hilbert-Schmidt
  and $W(t)$ takes values in $U$.

 Let $\cale$ be a polish space, and for every
 $\eps>0$,
 $\calg^\eps: C([0,T], U) \to \cale$ be a
 measurable map. Denote by
 \be\label{pre0}
 X^\eps = \calg ^\eps (W),
 \quad \forall \ \eps>0.
\ee
  We will investigate the large deviation principle
  of $X^\eps$ as $\eps \to 0$.
  To that end, we first recall some notation from
  \cite{bud1}.

Given $N>0$, denote  by
  \be\label{pre0a}
S_N
=\{v\in L^2(0,T; H): \int_0^T \| v(t)\|_H^2 dt
\le N\}.
\ee
Then $S_N$ is a polish space
endowed with the weak topology.
Throughout this paper, we always assume
$S_N$ is equipped with the weak topology,  unless
otherwise stated.
Let $\cala$
 be the space of all
 $H$-valued stochastic processes
 $v$
 which are progressively measurable
 with respect to $\{\calf_t\}_{t\in [0,T]}$
 and
 $\int_0^T \| v (t)\|^2 dt<\infty$
 $P$-almost surely.
 Denote by
  \be\label{pre0b}
 \cala_N
 =\{v\in \cala: v (\omega) \in S_N
 \ \text{for almost  all } \omega \in \Omega
 \}.
\ee
  
 The large deviation principle
 of  the family $\{X^\eps\}$
 is concerned with the exponential decay
 of distributions  of
 $\{X^\eps\}$
 on  tail events as $\eps \to 0$.
 Such exponential decay
 is  characterized by a rate function 
 $I: \cale \to [0, \infty]$.
 
 \begin{defn}
 A function $I : \cale \to [0, \infty]$
 is called a rate function
 on $\cale$  if
 it is lower  semi-continuous in $\cale$.
 A rate function $I$ on $\cale$ is 
 said to be a good
 rate function on $\cale$ if for every
 $0\le C <\infty $, the level set
 $\{x\in \cale: I(x) \le C\}$ is a compact
 subset of $\cale$.
 \end{defn}
 
 \begin{defn}
 The family $\{X^\eps\}$ is said to satisfy the
 large deviation principle
 in $\cale$ with a
 rate function $I:\cale \to [0,\infty]$ if 
 for every  Borel subset $B$ of $\cale$,
 $$
 -\inf_{x\in B^\circ}
  I(x)
  \le \liminf _{\eps \to 0}
  \eps \log 
  P(X^\eps \in B )
  \le
     \limsup  _{\eps \to 0}
  \eps \log 
  P(X^\eps \in B )
  \le
   -\inf_{x\in \overline{B} }
  I(x).
  $$
  where $B^\circ$ and $\overline{B}$ are
  the interior and the closure of $B$ in $\cale$,
  respectively.
 \end{defn}
 
 Since $\cale$ is a polish space,
 it is well known  that
 the family $\{X^\eps\}$
 satisfies the  
 large deviation principle
 on $\cale$ with a
 rate function $I:\cale \to [0,\infty]$ 
 if and only if
 $\{X^\eps\}$
 satisfies the  
 Laplace  principle
 on $\cale$ with the same 
 rate function.
 The concept of    Laplace principle of 
 $\{X^\eps\}$ is given    below.

  \begin{defn}
 The family $\{X^\eps\}$ is said to satisfy the
 Laplace  principle
 in  $\cale$ with a
 rate function $I:\cale \to [0,\infty]$ if 
 for all bounded  and
 continuous  $H: \cale \to \R$, 
 $$
\lim _{\eps \to 0}
  \eps \log 
  \E \left (
  e^{-{\frac 1\eps} H(X^\eps) }
  \right )
  =
  -\inf_{x\in \cale}
  \left \{
  H(x) +I(x)
  \right \}.
  $$
  \end{defn}
  
  In order to prove the large deviation principle
  of $X^\eps$, we will assume that
  the family $\{\calg^\eps\}$  fulfills the following
  conditions: 
  there exists a measurable map
  $\calg^0:    C([0,T], U) \to \cale$
  such that
  \begin{enumerate}
  \item[(\bf {H1})] \    If $N<\infty$ and
  $\{v^\eps\}\subseteq \cala_N$
  such that $\{v^\eps\}$
  converges  in distribution
  to $v$ as 
  $S_N$-valued random variables,
  then
  $\calg^\eps \left (
  W +\eps^{-\frac 12} \int_0^{\cdot}
  v^\eps (t) dt
  \right )$
  converges in distribution
  to 
    $\calg^0 \left (
    \int_0^{\cdot}
  v  (t) dt
  \right )$.
  
  \item[(\bf {H2})]   \  For every $N<\infty$, the
  set  $\left \{
  \calg^0 (\int_0^{\cdot} v(t) dt):\
  v \in S_N
  \right \}$
  is a compact subset of $\cale$.
    \end{enumerate}

  Define $I: \cale \to [0, \infty]$ by,
  for every $x\in \cale$,
 \be\label{pre1}
  I(x)= \inf 
  \left \{
  {\frac 12} \int_0^T \| v(t)\|^2_H dt:\
   v\in L^2(0,T; H)
   \ \text{such that}\  
  \calg^0 \left (\int_0^\cdot v(t) dt\right ) =x
  \right \},
 \ee
  with the convention that
  the infimum over an  empty set
   is taken to be $\infty$.
  By assumption ${\bf (H2)}$, we find that
  every level set of the map  $I$
  as defined by \eqref{pre1} is a compact subset
  of $\cale$,
  which further implies 
  the lower semi-continuity of $I$.
 By definition,    this map $I$ is a good
  rate function on $\cale$.
  Moreover, under ${\bf (H1)}$ and
  ${\bf (H2)}$, 
  the family $\{X^\eps\}$ satisfies the Laplace
  principle in $\cale$ with rate function $I$
  as stated below
  (see, \cite[Theorem 4.4]{bud1}).
  
  \begin{prop}\label{LP1}
  If $\{\calg^\eps\}$ satisfies
    {\rm {({\bf H1})-({\bf H2})}}, then
    the family
    $\{X^\eps\}$ as given by \eqref{pre0}
    satisfies  
    the Laplace
  principle in $\cale$ with rate function $I$
  as defined by \eqref{pre1}.
  \end{prop}

\section{Well-posedness of stochastic lattice systems} 
\setcounter{equation}{0}
 
 In this section,  we 
 discuss the existence  and uniqueness
 of solutions  to system \eqref{intr1}-\eqref{intr2},
 which is  needed for establishing the
 large deviation principle of the solutions.

   Let $(\Omega, \mathcal{F}, 
   \{ { \mathcal{F}} _t\} _{t\in \R},  P )$
be a  complete filtered probability 
space satisfying the usual condition, and
   $ (W_k)_{k\in \N}$  is a 
sequence of  independent  real-valued 
standard Wiener 
processes  
defined on  $(\Omega, \mathcal{F}, 
\{ { \mathcal{F}} _t\} _{t\in \R},  P )$.
Denote by   
$$
 \ell^2 = \left \{
 u = (u_i)_{i \in \Z} \ : \  \sum\limits_{i \in \Z} | u_i
 |^2 <   \infty \right \} 
 $$
 with   norm $ \| \cdot \|$ and  inner product
 $ (\cdot , \cdot )$,  respectively. 
 Recall that  the negative discrete
 $N$-dimensional 
 Laplace operator $A:\ell^2\to \ell^2$
   is given by, for every $u=(u_i)_{i\in \Z}\in \ell^2$
   and $i=(i_1,i_2,\dots,i_N)\in\Z$,
$$
    (Au)_i=-u_{(i_1-1,i_2,\dots,i_N)}
 -u_{(i_1,i_2-1,\dots,i_N)} -\dots-u_{(i_1,i_2,\dots,i_N-1)}
  $$
  $$
  +2Nu_{(i_1,i_2,\dots,i_N)}
-u_{(i_1+1,i_2,\dots,i_N)}-u_{(i_1,i_2+1,\dots,i_N)}
 -\dots-u_{(i_1,i_2,\dots,i_N+1)}\big).
$$
For convenience, for every $j=1,\ldots, N$,  
define $A_j, B_j, B^*_j :\ell^2\to \ell^2$
  by, for   $u=(u_i)_{i\in \Z}\in \ell^2$
   and $i=(i_1,i_2,\dots,i_N)\in\Z$,
   $$
   (A_ju)_i= -u_{(i_1,\dots,i_j+1,\dots,
 i_N)}+2u_{(i_1,\dots,i_j,\dots,i_N)}-u_{(i_1,\dots,i_j-1,\dots,i_N)},
   $$ 
   $$
    (B_ju)_i=u_{(i_1,\dots,i_j+1,\dots, i_N)}
-u_{(i_1,\dots,i_j,\dots,i_N)},
$$
and 
   $$
     (B_j^{*}u)_i=u_{(i_1,\dots,i_j-1,\dots, i_N)}-u_{(i_1,\dots,i_j,\dots,i_N)}.
     $$
Then we have,
$$
A=\sum_{j=1}^N A_j,
\quad   A_j=B_jB_j^{*}=B_j^{*}B_j.
$$
 Given $u=(u_i)_{i\in \Z}\in \ell^2$, denote by
 $$
 Bu =(B_1u,\ldots, B_N u)
 \quad \text{and}
 \quad
 \| Bu\|= \left (
 \sum_{j=1}^N \| B_j u\|^2
 \right )^{\frac 12} .
 $$

 In the sequel,   we   assume 
 that  for every $i\in \Z$,
 the nonlinear function
 $F_i:\R^+ \times \R
\to \R $ satisfies,
 \be\label{F0}
 F_i (t,s )  = F_0 (s) 
    - g_i (t) ,
    \quad \forall \ t\in \R^+, \ s\in \R,
  \ee
    where  
$F_0:  \R
\to \R $ is  continuously differentiable  and 
     there exists $\gamma \le 0 $ such that
\be\label{F1}
F_0( 
0) =0  \quad \text{and} 
\quad    F_0^\prime  (  s) 
   \ge  \gamma
\ \text{  for  all } \ s\in \R.
\ee 
Note that  for every $k\in\N$,
the classical nonlinearity
$F_0(s)= s^{2k+1} -s$ for the reaction-diffusion
equation indeed satisfies condition \eqref{F1}.

 For the sake of convenience,  we  write
$$  
f(s) = F_0(
s) -\gamma s \   \text{ for all } 
    s \ \in \R.
$$
It follows from \eqref{F1} that
\be\label{f1}
f( 0) = 0 \ \text{ and } \ 
  f^\prime (s) 
  \ge 0 \ \text{  for all } 
s\in \R.
\ee

Given  $ u=(u_i)_{i \in \Z} \in \ell^2$, denote by
 $ {f} (u) =  \left ( f( u_i) \right )_{i \in \Z}$.
 By \eqref{f1} we find  that   
  for every $R>0$, 
  there exists a constant $L_R>0  $
 such that 
 for all $u, v \in \ell^2$ with $\|u \| \le R$
 and $\| v\| \le R$,
\be\label{f1a1}
 \| {f} ( u)- {f} ( v)\|\leq
L_R   \|u-v\|. 
\ee
Moreover, 
\be\label{f2}
(f(u)-f(v), u-v) \ge 0\quad \text{for all } \ 
u,\ v \in \ell.
\ee

For the nonlinear diffusion term,
we assume that
for every $i\in \Z$  and $k\in \N$,
$\sigma_{k,i}: \R^+ \times \R
\to \R$ satisfies
\be\label{sigma0}
\sigma_{k,i}(t,s)
= 
 h_{k,i} (t) +\delta_{k,i} \sigma_k^0 (s),
 \quad \forall \ t\in \R^+, \ s\in \R,
\ee
 where   
$\sigma^0_k: \R \to \R$  is  locally Lipschitz  continuous
in the sense that
  for any  bounded interval $I$,  there exists
a constant $L_I>0$  such that 
\be\label{sigma1}
|\sigma^0_k (s_1) -\sigma^0_k (s_2) |
\le L_I |s_1 -s_2|, \quad \text{for all } \ s_1, s_2
\in I,   \  k\in \N.
\ee
We further 
  assume that there exists 
   $\alpha>0$ such that
\be\label{sigma2}
|\sigma^0_k (s)| \le \alpha(1 + |s|), \ \ \text{ for all } \ s\in \R
\quad \text{and } \  \   k\in \N.
\ee 
 For   the   sequence $\delta
=(\delta_{k,i})_{k\in \N, i\in \Z}$
in \eqref{sigma0} we  
      assume 
\be\label{delta1}
\| \delta\|^2 =
  \sum_{k\in \N}\sum_{i\in \Z} | \delta_{k,i}|^2 <\infty.
\ee
  For each $k\in \N$, define  
an operator $\sigma_k: \ell^2 \to \ell^2$ by
\be
\label{defsigma}
\sigma_k  (u) = (\delta_{k,i} \sigma^0_k(u_i))_{i\in \Z},
\quad \text{for all } \ u=(u_i)_{i\in \Z} \in \ell^2 .
\ee
Then by  
 \eqref{sigma2}, \eqref{delta1}
 and \eqref{defsigma}  we see that  
\be\label{sigrow1}
\sum_{k\in \N}
\| \sigma_k (u) \|^2
\le 
2 \alpha^2  \| \delta \|^2
 (1+ \| u \|^2), \quad  \text{ for  all }  u\in \ell^2.
\ee 
Furthermore,  by \eqref{sigma1},
we find that 
 for every $R>0$,  there exists 
$L_R >0$    such that
for all $u, v \in \ell^2$ with $\| u\|\le R$
and $\|v\|\le R$,
\be\label{sigrow2}
\sum_{k\in \N}
\| \sigma_k (u) -\sigma_k (v) \|^2
\le L_R \| \delta \|^2 \| u-v\|^2.
\ee

For the sequences
  $g(t)=(g_i(t))_{i\in \Z}$
  in \eqref{F0}
and $h_k(t)=(h_{k,i}(t))_{i\in \Z}$
in \eqref{sigma0}, we assume that  
for every $T>0$,
\be\label{gh}
\int_0^T 
  \|g(t)\|^2 dt <\infty
  \quad \text{and}
  \quad
    \sum_{k=1}^\infty  
 \| h_k  \|^2 _{L^\infty(0,T;  \ell^2 )}
 <\infty .
\ee

  Note that system \eqref{intr1}-\eqref{intr2}
is equivalent to the 
It\^{o}  stochastic equation
for  
  $u^\eps=(u^\eps_i)_{i\in\Z}$ in $\ell^2$:
 \be\label{intr3}
 d u^\eps (t)
 + \nu Au^\eps (t)  dt
 +f(u^\eps(t) ) dt 
+ \gamma  u^\eps(t) dt  =
g (t)  dt
 + \sqrt{\eps}
  \sum_{k=1}^\infty \left (
 h_k (t) + \sigma_k (u^\eps (t))
 \right ) dW_k(t),
\ee
 with  initial data
 \be\label{intr4}
 u^\eps(0)=u_0\in\ell^2.
 \ee

 Under conditions
 \eqref{F1}, \eqref{sigma1}-\eqref{delta1}
 and \eqref{gh}, one can show that
 system \eqref{intr3}-\eqref{intr4}
 is well-posed in
 $\ell^2$ for every $\eps\in (0,1)$
 (see, e.g.,   \cite[Theorem 2.5]{wan2019});
 more precisely,
 for every
  $u_0 \in L^2(\Omega,\calf_0; 
   \ell^2) $, there exists a unique 
  continuous  $\ell^2$-valued  $\calf_t$-adapted
 stochastic process $u^\eps$   such that
  $u^\eps\in L^2(\Omega, C([0, T], \ell^2))$  for  all $T>0$, 
  and for almost all $\omega \in \Omega$,
   $$
  u^\eps (t)
 =u_0 +  \int_0^t 
 \left (-\nu Au^\eps (s) 
 -f(u^\eps(s) ) 
 -  \gamma  u^\eps(s)  + g (s)   \right ) ds
 +  \sqrt{\eps} \sum_{k=1}^\infty 
 \int_0^t  \left (
 h_k (s) + \sigma_k (u^\eps (s))
 \right ) dW_k  
$$ 
in $\ell^2$  
for all $t \ge 0$.
Moreover,  for every $T>0$,
there exists a positive number
$C=C(T)$ independent of $u_0$ 
and $\eps \in (0,1)  $ such that
\be\label{uniest1}
\E \left ( 
\| u^\eps \|^2_{C([0,T],  \ell^2) }  \right ) \le 
C
\left (  \E (\| u_0 \|^2)
+  1
+ 
\int_0^T 
(\| g(s) \|^2
+\sum_{k=1}^\infty \| h_k(s) \|^2 ) ds
\right ).
\ee

 For convenience, we set
 $$
 H=\{ u=(u_j)_{j=1}^\infty:
 \sum_{j=1}^\infty |u_j|^2<\infty
 \}.
 $$
 For every $k\in \N$,
 let $e_k =(\delta_{k,j})_{j=1}^\infty$ 
 with $\delta_{k,j}=1$  for $j=k$ and
 $\delta_{k,j}=0$ otherwise.
 Then $\{e_k\}_{k=1}^\infty$ is an orthonormal
 basis of $H$.
 Let $I$ be the identity operator on $H$ and $W$
 be the cylindrical Wiener process
 in $H$ with covariance operator $I$ as given by
 $$
 W(t)
 =\sum_{k=1}^\infty W_k(t) e_k,
 \quad t\in \R^+,
 $$
  where
  the series
  converges
  in $L^2(\Omega, \calf; C([0,T], U))$
  for every $T>0$ with $U$ being a separable Hilbert
  space such that the embedding
  $H\hookrightarrow U$ is a Hilbert-Schmidt operator.
  
  Given $u\in \ell^2$  and $t\ge 0$, define
  $\sigma (t, u): H \to \ell^2$  by
  \be\label{defsigma2}
  \sigma (t,u) (v)  =  
 \sum_{k=1}^\infty 
  \left (
 h_k (t) + \sigma_k (u)
 \right )v_k,
 \quad \forall \ v=(v_k)_{k=1}^\infty \in H.
\ee
 Note that
 the series in \eqref{defsigma2}
 is convergent  in $\ell^2$ by
 \eqref{sigrow1}  and \eqref{gh}.
 Moreover, the operator
 $\sigma (t, u): H \to \ell^2$
 is Hilbert-Schmidt   and 
  \be\label{defsigma3}
 \| \sigma (t, u)\|_{L  (H, \ell^2)}
 \le
   \| \sigma (t, u)\|_{L_2 (H, \ell^2)}
 = \left (
 \sum_{k=1}^\infty
 \| h_k (t) + \sigma_k (u) \|^2
 \right )^{\frac 12}<\infty.
\ee
Hereafter, we use
${L  (H, \ell^2)}$
to denote the   space
of bounded linear operators from
$H$ to $\ell^2$ with norm
$\| \cdot\|_{L  (H, \ell^2)}$,
and use
${L _2 (H, \ell^2)}$
to denote the   space
of Hilbert-Schmidt  operators from
$H$ to $\ell^2$ with norm
$\| \cdot\|_{L_2  (H, \ell^2)}$.

 With above notation,
   system \eqref{intr3}-\eqref{intr4}
   can be reformulated as 
 \be\label{intr5}
 d u^\eps (t)
 + \nu Au^\eps (t)  dt
 +f(u^\eps(t) ) dt 
+ \gamma  u^\eps(t) dt  =
g (t)  dt
 + \sqrt{\eps} \sigma(t, u^\eps(t))  
   dW  (t),
\ee
 with  initial data
 \be\label{intr6}
 u^\eps(0)=u_0\in\ell^2.
 \ee
  
  In the next section, we  examine the
  large deviation principle of \eqref{intr5}-\eqref{intr6}.

 \section{Large deviation principles of lattice systems}
 
 In this section, we prove the large
 deviation principle of the 
 family of  solutions  $\{u^\eps\}  $ 
 of \eqref{intr5}-\eqref{intr6}
 as $\eps \to 0$ which is stated below.

 \begin{thm}\label{main_resu}
 Suppose that
 \eqref{F1}, \eqref{sigma1}-\eqref{delta1}
 and \eqref{gh} hold, and $u^\eps$
 is the solution of 
   \eqref{intr5}-\eqref{intr6}.
   Then the family
   $\{u^\eps\}$,  as $\eps \to 0$, 
satisfies the large deviation principle
in $C([0,T], \ell^2)$ with the good rate function
as given by \eqref{rate_LS}.
\end{thm}
  
  The rest of the paper is devoted to the proof
  of Theorem \ref{main_resu}. 
 Note that  for every  $\eps\in (0,1)$ and $T>0$,
  by the existence and uniqueness of solution
  to \eqref{intr5}-\eqref{intr6}
  in $C([0,T], \ell^2)$,
  there exists a Borel-measurable map
  $\calg^\eps: C([0,T], U) \to C([0,T], \ell^2)$
  such that
  $$
  u^\eps =\calg^\eps (W),
  \quad  \text{P-almost surely}.
  $$
  To study the large deviation principle of
  $\{u^\eps\}$ as $\eps \to 0$, we
  introduce a deterministic  control
  system corresponding to \eqref{intr5}.
  Given a control  $v\in L^2(0,T; H)$, 
  solve  for $u_v$ in terms of   the controlled  equation:
   \be\label{contr1}
{\frac {d u_v (t)}{dt}}
=-  \nu Au_v  (t)  
-f(u_v (t) )  
- \gamma  u_v (t)   
+
g (t)  
 +   \sigma(t, u_v (t))  
    v (t),
\ee
 with  initial data
 \be\label{contr2}
 u_v  (0)=u_0 \in\ell^2.
 \ee
  
As usual, by a solution
$u_v$  to \eqref{contr1}-\eqref{contr2}
on $[0,T]$, we mean  
$u_v\in C([0,T], \ell^2)$  such that
for all $t\in [0,T]$,
 \be\label{contr3}
u_v(t) = u_0
+ \int_0^t  
\left (
-\nu Au_v  (s)   
-f(u_v (s) )  
- \gamma  u_v (s)    
+ g (s)   
 +   \sigma(s, u_v (s))  
    v (s)
  \right ) ds.
\ee

 We first prove the existence and
 uniqueness of solutions  to
 \eqref{contr1}-\eqref{contr2}
 in the sense of \eqref{contr3}.

 \begin{lem}\label{exis_sol}
 Suppose that
 \eqref{F1}, \eqref{sigma1}-\eqref{delta1}
 and \eqref{gh} hold.
 Then for every 
  $v\in L^2(0,T; H)$,
   problem
\eqref{contr1}-\eqref{contr2}
 has a unique solution
 $u_v \in C([0,T], \ell^2)$.

 Furthermore,
     for each $R_1>0$  and $R_2>0$,
    there  exists
       $C_1 =C_1  (R_1,R_2, T)>0$ such that  
       for any $u_{0,1}, u_{0,2}
       \in \ell^2$ with 
    $\| u_{0,1} \|\le R_1,
    \| u_{0,2}  \|\le R_1$,  and any
    $v_1, v_2\in L^2(0, T; H)$
    with
    $\| v_1\|_{L^2(0, T; H)}\le R_2$
    and $\| v_2\|_{L^2(0, T; H)}\le R_2$,
    the solutions $u_{v_1}$ and $u_{v_2}$
    of  \eqref{contr1}-\eqref{contr2}
    with initial data $u_{0,1}$
    and $u_{0,2}$,   respectively,   
    satisfy
    \be\label{exis_sol 1}
    \| u_{v_1}-u_{v_2}\|_{C([0,T], \ell^2)}^2
    \le
    C_1
    \left (
    \| u_{0,1}- u_{0,2}\|^2
    +\|  v_1- v_2 \|^2_{L^2(0, T; H)}
    \right ),
   \ee
   and
    \be\label{exis_sol 2}
    \| u_{v_1} \|_{C([0,T], \ell^2)}^2
    \le
    C_1 .
   \ee
 \end{lem}
 
 \begin{proof}
 Let   $v\in L^2(0,T; H)$ be given.
 We first prove the existence and uniqueness of
 solution to system \eqref{contr1}-\eqref{contr2}.
  By dropping
   the subscript $v$, system \eqref{contr1}-\eqref{contr2}
   can be written as
   \be\label{exis_solp1}
   {\frac {du(t)}{dt}}
   = G(t, u(t)),\quad  u(0) =  u_0,
   \ee
   where 
   \be\label{exis_solp2}
 G(t,u)
 =  -  \nu Au   
-f(u    )  
- \gamma  u   
+
g (t)  
 +   \sigma(t, u )  
    v (t).
   \ee
    Since  $v\in L^2 (0,T; H)$,
    by 
   \eqref{f1}, \eqref{f1a1} and \eqref{sigrow1} we find
   from \eqref{exis_solp2} that 
   for every $R>0$, there exists a constant
   $c_1=c_1(R)>0$ such that
   for all $ t \in [0,T]$  and $u\in \ell^2$ with
   $\| u \| \le R$,
    \be\label{exis_solp3}
   \| G(t, u) \|^2 \le c_1 
   \left (1+ \| g(t)\|^2
   + \|u\|^2
   \right )
   + c_1\left (1
   +\sum_{k=1}^\infty \| h_k(t)\|^2
   + \| u \|^2
   \right ) \| v(t) \|^2_H.
   \ee
Similarly,  by  \eqref{f1a1}
and \eqref{sigrow2} we find that
there exists 
    $c_2=c_2(R)>0$ such that
   for all $ t \in [0,T]$  and $u_1, u_2 \in \ell^2$ with
   $\| u_1 \| \le R$
   and
     $\| u_2 \| \le R$,
       \be\label{exis_solp4}
   \| G(t, u_1) -
   G(t, u_2) 
   \|^2 \le c_2 
  \left   (1+ \| v(t)\|^2_H
  \right )
     \| u_1-u_2 \|^2.
   \ee
   
   It follows from \eqref{exis_solp3}-\eqref{exis_solp4}
   that for each $u_0\in \ell^2$,
   equation \eqref{exis_solp1}
   has a unique local maximal solution
   $u\in C([0,T_0), \ell^2 )$
    for some $0< T_0 \le T$.
   Next, we show this solution is actually
   defined on  the entire interval $[0,T]$
   by uniform estimates of solutions.
   
   By \eqref{exis_solp1} we have
   for all $t\in (0, T_0)$,
   $$
   {\frac 12} {\frac d{dt}}
   \| u(t) \|^2
   =-\nu \| Bu(t)\|^2 -(f(u(t)), u(t))
   -\gamma \| u(t) \|^2
   $$
   $$
   + (g(t), u(t))
   +(\sigma (t, u(t)) v(t), u(t)),
   $$
   which along with \eqref{f1} and \eqref{f2}
   implies that
  \be\label{exis_solp5}
    {\frac d{dt}}
   \| u(t) \|^2
   \le 
   - 2\gamma \| u(t) \|^2
   +  2 (g(t), u(t))
   + 2 (\sigma (t, u(t)) v(t), u(t)).
\ee
   By Young\rq{}s inequality we have
 \be\label{exis_solp6} 
  2 |(g(t), u(t))|
  \le \| u(t) \|^2 + \| g(t) \|^2.
  \ee
  For the last term
  on the right-hand side of \eqref{exis_solp5},
  by \eqref{defsigma3}
  and \eqref{sigrow1} we get
 $$
  2| (\sigma (t, u(t)) v(t), u(t))|
  \le \| \sigma (t, u(t)) v(t)\|^2
  + \| u(t) \|^2
  $$
  $$ 
  \le \| \sigma (t, u(t))\|^2_{L(H,\ell^2)} \| v(t)\|^2_H
  + \| u(t) \|^2
  $$
   $$ 
  \le
  \sum_{k=1}^\infty
  \| h_k(t) + \sigma_k (u(t)) \|^2  \| v(t)\|^2_H
  + \| u(t) \|^2
  $$   
   $$ 
  \le
  2 \| v(t)\|^2_H\sum_{k=1}^\infty
  \| h_k(t)\|^2
  +   2 \| v(t)\|^2_H\sum_{k=1}^\infty
     \|  \sigma_k (u(t)) \|^2   
  + \| u(t) \|^2
  $$
     \be\label{exis_solp7} 
  \le
  2 \| v(t)\|^2_H\sum_{k=1}^\infty
  \| h_k(t)\|^2
   +
  4  \alpha^2 \| \delta\|^2
   (1+ \| u(t) \|^2) 
   \| v(t)\|^2_H
  + \| u(t) \|^2.
  \ee
   By \eqref{exis_solp5}-\eqref{exis_solp7} we get
   for all $t\in (0,T_0)$,
    $$
    {\frac d{dt}}
   \| u(t) \|^2 
   \le 
   \left (2- 2\gamma 
   + 4\alpha^2
   \| \delta\|^2    \| v(t)\|^2_H
   \right )
   \| u(t) \|^2
   $$
    \be\label{exis_solp8}
   +   4  \alpha^2 \| \delta\|^2
   \| v(t)\|^2_H
   + \| g(t) \|^2
   +
    2 \| v(t)\|^2_H\sum_{k=1}^\infty
  \| h_k(t)\|^2.
  \ee
   By \eqref{exis_solp8} we find that
  for all $t\in [0, T_0)$,
  $$
   \| u(t) \|^2 
   \le 
    e^{  \int_0^t  (2- 2\gamma 
   + 4\alpha^2
   \| \delta\|^2    \| v (r)\|^2 _H
     )  dr }\|u_0\|^2
     $$
     $$
     +
      4  \alpha^2 \| \delta\|^2
    \int_0^t
     e^{  
      \int_s^t  (2- 2\gamma 
   + 4\alpha^2
   \| \delta\|^2    \| v (r)\|^2_H 
     )  dr 
      }  \| v(s) \|^2_H
      ds
   $$
    $$
     + 
 \int_0^t
     e^{  
      \int_s^t  (2- 2\gamma 
   + 4\alpha^2
   \| \delta\|^2    \| v (r)\|^2_H 
     )  dr 
      }   \| g(s) \|^2  ds
   $$
 $$
 + 
    2 \sum_{k=1}^\infty \| h_k\|^2_ {L^\infty
     (0,T; \ell^2 )}
     \int_0^t
     e^{  
      \int_s^t  (2- 2\gamma 
   + 4\alpha^2
   \| \delta\|^2    \| v (r)\|^2_H 
     )  dr 
      }\| v(s) \|^2_ H ds 
       $$
        $$ 
   \le 
    e^{   (2- 2\gamma)T
    + 
     4\alpha^2
   \| \delta\|^2   \int_0^T  \| v (r)\|^2 _H
       dr }\|u_0\|^2
     $$
           $$ 
  +
    4  \alpha^2 \| \delta\|^2
      e^{   (2- 2\gamma)T
    + 
     4\alpha^2
   \| \delta\|^2   \int_0^T  \| v (r)\|^2 _H
       dr }
       \int_0^T 
        \| v(s) \|^2_H
      ds
      $$
         $$ 
  + 
      e^{   (2- 2\gamma)T
    + 
     4\alpha^2
   \| \delta\|^2   \int_0^T  \| v (r)\|^2 _H
       dr }
       \int_0^T 
        \| g(s) \|^2 
      ds
      $$
        \be\label{exis_solp9}
 + 
    2
     e^{   (2- 2\gamma)T
    + 
     4\alpha^2
   \| \delta\|^2   \int_0^T  \| v (r)\|^2 _H
       dr }
     \sum_{k=1}^\infty \| h_k\|^2_{L^\infty
     (0,T; \ell^2 )}
     \int_0^T
     \| v(s) \|^2_ H ds .
    \ee
    By  \eqref{gh} and \eqref{exis_solp9} we infer that
    for each $R_1>0$  and $R_2>0$,
    there  exists
       $c_3=c_3(R_1,R_2, T)>0$ such that 
       for any $u_0\in \ell^2$ with
       $\| u_0 \|\le R_1$ and any
        $v\in L^2(0, T; H)$ with
        $\| v\|_{L^2(0, T; H)} \le R_2$,
        the solution $u$ satisfies 
       \be\label{exis_solp10}
   \| u(t) \|^2 
   \le   c_3,\quad \forall \ t\in[0,T_0),
\ee
   which implies that    
   $T_0 =T$ and hence
  the solution  $u$ of  \eqref{exis_solp1}
  is defined on the entire interval $[0,T]$.

   Next, we prove  \eqref{exis_sol 1}.
    Let $v_1, v_2$ be given in 
    $ L^2(0, T; H)$,
    and denote by
    $$
    u_1=u_{v_1}
    \quad \text{ and } \quad
    u_2=u_{v_2}.
    $$
    Suppose 
    $\| u_{0,1} \|\le R_1$,  $\| u_{0,2} \|\le R_1$, 
    $\| v_1\|_{L^2(0, T; H)}\le R_2$
    and $\| v_2\|_{L^2(0, T; H)}\le R_2$.
    Then by \eqref{exis_solp10} we have
    \be
    \label{exis_solp11}
      \| u_1 (t) \| +   \| u_2 (t) \|
   \le   c_4,\quad \forall \ t\in[0,T],
    \ee
    where $c_4=c_4(R_1, R_2, T)>0$.
    Note that
    \eqref{exis_solp11} implies
      \eqref{exis_sol 2}.

   By \eqref{exis_solp1}-\eqref{exis_solp2}
     we have
     $$
     {\frac d{dt}}
     \| u_1(t) -u_2(t)\|^2
     =-2 \nu \| B(u_1(t)-u_2(t))\|^2
     $$
     $$
     -2 (f(u_1(t)) -f(u_2(t)), u_1(t)-u_2(t))
     -2\gamma \| u_1(t)-u_2(t )\|^2
     $$
     $$
     +2
     \left (
     \sigma (t, u_1(t))v_1(t)-
     \sigma (t, u_2(t))v_2(t),
     \  u_1(t)-u_2(t)
     \right ),
     $$
  which together with \eqref{f2}
  gives  
      $$
     {\frac d{dt}}
     \| u_1(t) -u_2(t)\|^2
     \le
     -2\gamma \| u_1(t)-u_2(t )\|^2
     $$
         \be\label{exis_solp12}
     +2
     \left (
     \sigma (t, u_1(t))v_1(t)-
     \sigma (t, u_2(t))v_2(t),
     \  u_1(t)-u_2(t)
     \right ).
   \ee
   For the last term in \eqref{exis_solp11},
   by \eqref{defsigma2}
    we get
   $$
   2
     \left (
     \sigma (t, u_1(t))v_1(t)-
     \sigma (t, u_2(t))v_2(t),
     \  u_1(t)-u_2(t)
     \right )
     $$ 
        $$
   \le 2
     \| 
     \sigma (t, u_1(t))v_1(t)-
     \sigma (t, u_2(t))v_2(t)\|
     \| u_1(t)-u_2(t) \|
     $$ 
  $$
   \le  2
     \| \sum_{k=1}^\infty
     (\sigma_k (u_1(t))v_{1,k} (t)
     -
     \sigma_k (u_2(t))v_{2,k} (t) )\| 
     \| u_1(t)-u_2(t) \|
     $$ 
       $$
  +  2
     \| \sum_{k=1}^\infty
     h_k(t) (v_{1,k} (t) - v_{2,k} (t) ) \| 
     \| u_1(t)-u_2(t) \|
     $$ 
      $$
 \le  2
     \| \sum_{k=1}^\infty
     (\sigma_k (u_1(t)) 
     -\sigma_k (u_2(t))  )
     v_{1,k} (t) \| 
     \| u_1(t)-u_2(t) \|
     $$ 
       $$
+  2
     \| \sum_{k=1}^\infty
      \sigma_k (u_2(t))  
     (v_{1,k} (t)  -
     v_{2,k} (t) )
      \| 
     \| u_1(t)-u_2(t) \|
     $$ 
      $$
  +  2
     \| \sum_{k=1}^\infty
     h_k(t) (v_{1,k} (t) - v_{2,k} (t) ) \| 
     \| u_1(t)-u_2(t) \|
     $$ 
       $$
 \le  2
     \left (
      \sum_{k=1}^\infty
     \| \sigma_k (u_1(t)) 
     -\sigma_k (u_2(t))  \|^2
     \right )^{\frac 12}
    \| v_{1 } (t) \|_H 
     \| u_1(t)-u_2(t) \|
     $$ 
      $$
+  2
    \left (
     \sum_{k=1}^\infty
      \| \sigma_k (u_2(t)) \|^2
      \right )^{\frac 12} 
    \| v_{1} (t)  -
     v_{2} (t)  \|_H
     \| u_1(t)-u_2(t) \|
  $$
          \be\label{exis_solp13}
+  2
    \left (
     \sum_{k=1}^\infty
      \| h_k(t)  \|^2
      \right )^{\frac 12} 
    \| v_{1} (t)  -
     v_{2} (t)  \|_H
     \| u_1(t)-u_2(t) \|.
     \ee
     By    \eqref{sigrow2} and
     \eqref{exis_solp11} we see that
     there exists $c_5=c_5 (R_1, R_2, T)>0$ such that
     for all $t\in [0, T]$,
    \be\label{exis_solp14}
     \left (
      \sum_{k=1}^\infty
     \| \sigma_k (u_1(t)) 
     -\sigma_k (u_2(t))  \|^2
     \right )^{\frac 12}
     \le c_5 \| \delta\| \|u_1(t) - u_2 (t)  \|.
\ee
On  the other hand, 
     by  \eqref{sigrow1}
     and  \eqref{exis_solp11}
      we know  that
     there exists $c_6=c_6 (R_1, R_2, T)>0$ such that
     for all $t\in [0, T]$,
    \be\label{exis_solp15}
      \left (
     \sum_{k=1}^\infty
      \| \sigma_k (u_2(t)) \|^2
      \right )^{\frac 12} 
      \le \alpha \| \delta \| c_6.
      \ee
      By   \eqref{exis_solp13}-\eqref{exis_solp15}
      we get  for all $t\in [0, T]$,
       $$
   2
     \left (
     \sigma (t, u_1(t))v_1(t)-
     \sigma (t, u_2(t))v_2(t),
     \  u_1(t)-u_2(t)
     \right )
     $$ 
     $$
     \le
     2  c_5 \| \delta\| \| v_1(t) \|_H  \|u_1(t) - u_2 (t)  \|^2
     +
     2 c_6\alpha \| \delta \| \|v_1(t) - v_2 (t)  \|_H
      \|u_1(t) - u_2 (t)  \|
      $$
     $$
     \le
     \left (2+
     2  c_5 \| \delta\| \| v_1(t) \|_H 
      \right )
      \|u_1(t) - u_2 (t)  \|^2 
      $$
         \be\label{exis_solp16}
      +
   \left  (c_6^2  \alpha ^2\| \delta \|^2
     +  
     \sum_{k=1}^\infty
      \| h_k(t)  \|^2  
      \right ) 
      \|v_1(t) - v_2 (t)  \|_H^2.
        \ee
      By \eqref{exis_solp12}
      and \eqref{exis_solp16} we obtain 
      for all $t\in (0, T]$,
        $$
     {\frac d{dt}}
     \| u_1(t) -u_2(t)\|^2
     \le
      \left (2-2\gamma +
     2  c_5 \| \delta\| \| v_1(t) \|_H 
      \right )
      \|u_1(t) - u_2 (t)  \|^2 
      $$
        \be\label{exis_solp17}
      +
    \left  (c_6^2  \alpha ^2\| \delta \|^2
     +  
     \sum_{k=1}^\infty
      \| h_k(t)  \|^2  
      \right ) 
      \|v_1(t) - v_2 (t)  \|_H^2,
      \ee
     from which we get
     for all $t\in [0, T]$,
     $$
      \| u_1(t) -u_2(t)\|^2
      \le e^{
      \int_0^t 
         (2-2\gamma +
     2  c_5 \| \delta\| \| v_1(r) \|_H 
      )dr
      } \| u_{0,1}  -u_{0,2}   \|^2
      $$
      $$
            +
 \left  (c_6^2  \alpha ^2\| \delta \|^2
     +  
     \sum_{k=1}^\infty
      \| h_k  \|^2_{L^\infty(0,T; \ell^2)}  
      \right ) 
    \int_0^t
    e^{
      \int_s^t 
         (2-2\gamma +
     2  c_5 \| \delta\| \| v_1(r) \|_H 
      )dr
      }
      \|v_1(s) - v_2 (s)  \|_H^2
      ds
      $$
        $$
        \le e^{ 
         (2-2\gamma)T  +
     2  c_5 \sqrt{T} R_2  \| \delta\|  
      } \| u_{0,1}  -u_{0,2} \|^2
      $$
    \be\label{exis_solp18}
       +  \left  (c_6^2  \alpha ^2\| \delta \|^2
     +  
     \sum_{k=1}^\infty
      \| h_k  \|^2_{L^\infty(0,T; \ell^2)}  
      \right ) 
         e^{ 
         (2-2\gamma)T  +
     2  c_5 \sqrt{T}  R_2 \| \delta\|  
      }     \|v_1  - v_2    \|^2_{L^2(0,T; H)}.
    \ee
    Then   \eqref{exis_sol 1} follows
    from \eqref{exis_solp18} immediately.
 \end{proof}

 As a
  consequence of Lemma \ref{exis_sol},
 we find that the solution
 $u_v$ of \eqref{contr1}-\eqref{contr2}
 is continuous in $C([0,T], \ell^2)$
 with respect to
 initial data  $u_0$ in $\ell^2$
 and control  $v$ in $L^2(0,T; H)$.
 In the sequel,  we will further
 prove the continuity of $u_v$
 in  $C([0,T], \ell^2)$
  with respect to
    $v\in L^2(0,T; H)$ in   the weak
    topology of  $L^2(0,T; H)$,
    which is related to condition ({\bf H2})
    for the Laplace principle of the solutions
    of \eqref{intr5}-\eqref{intr6}.
    As a necessary step, we first prove the following
    convergence.

 \begin{lem}\label{wc_dif}
 Suppose that
 \eqref{sigma1}-\eqref{delta1}
 and \eqref{gh} hold.
For a fixed $\xi\in L^\infty(0,T; \ell^2)$,
define an operator
$\calt: L^2(0,T; H)
\to C([0,T],   \ell^2)$ by
\be\label{wc_dif 1}
\calt (v) (t)
=\int_0^t
\sigma (s, \xi (s)) v(s) ds,
\quad \forall \ v \in  L^2(0,T;H).
\ee
Then we have:
\begin{enumerate}
\item[(i)] $\calt$ is continuous from the
weak topology of  $L^2(0,T; H)$
to the strong topology
of  $ C([0,T],   \ell^2)$. 
\item[(ii)] $\calt: L^2(0,T; H)
\to C([0,T],   \ell^2)$ is compact with respect
to the strong topology of $ C([0,T],   \ell^2)$. 
\end{enumerate}
 \end{lem}

    \begin{proof}
    (i). 
    Note that the operator  $\calt: L^2(0,T; H)
\to C([0,T],   \ell^2)$  is well defined.
Indeed, by \eqref{sigrow1} and \eqref{defsigma3}
we have, for every $v\in  L^2(0,T; H)$,
$$
\int_0^T
\| \sigma (s, \xi (s)) v(s) \| ^2 ds
\le
\int_0^T
\| \sigma (s, \xi (s))\|_{L(H,\ell^2)}^2 
\| v(s) \| ^2_H ds
$$
$$
\le 2 
\int_0^T
\sum_{k=1}^\infty
\left (\| h_k(s) \|^2
+  
\| \sigma_k (\xi(s)) \| ^2
\right ) \| v(s)\|^2_H  ds 
$$
$$
\le   
\int_0^T \left ( 2
\sum_{k=1}^\infty
\| h_k(s) \|^2
+  4\alpha^2\|\delta\|^2
(1+ \|\xi(s)\|^2) 
\right )
  \| v(s)\|^2_H  ds 
$$
\be\label{wc_difp1}
\le   
  \left ( 2
\sum_{k=1}^\infty
\| h_k \|^2_{L^\infty(0,T; \ell^2)}
+  4\alpha^2\|\delta\|^2
(1+ \|\xi\|^2_{L^\infty(0,T; \ell^2)}) 
\right )
  \int_0^T \| v(s)\|^2_H  ds <\infty,
\ee
which implies that
$\calt(v)$ as given by \eqref{wc_dif 1}
belongs to $   C([0,T],   \ell^2)$
for all $v\in L^2(0,T; H)$.

It is evident 
that $\calt:  L^2(0,T; H)
\to    C([0,T],   \ell^2)$ is linear.
On the other hand, 
 By \eqref{wc_dif 1}  we have
 for all $v\in L^2(0,T; H)$,
$$
\| \calt (v) \|^2_{C([0,T],   \ell^2)}
\le
\left ( \int_0^T
\| \sigma (s, \xi (s)) v(s)\|  ds
\right )^2
\le
T   \int_0^T
\| \sigma (s, \xi (s)) v(s)\|^2  ds,
$$
which along with
\eqref{wc_difp1} shows that
 $\calt:  L^2(0,T; H)
\to    C([0,T],   \ell^2)$ 
    is bounded.

    Since 
    $\calt:  L^2(0,T; H)
\to    C([0,T],   \ell^2)$ is linear and continuous
in the strong topology, we know that
  $\calt:  L^2(0,T; H)
\to    C([0,T],   \ell^2)$
is also continuous in the weak topology;
 that is, 
  if   $v_n \to v$ weakly in 
  $L^2(0,T; H)$,
  then
  $ \calt(v_n)  \to \calt (v) $  weakly
  in   $ C([0,T],   \ell^2)$.
  Next, we prove actually
    $ \calt(v_n)   \to \calt (v) $  strongly
  in   $ C([0,T],   \ell^2)$  for which we need to
  verify:
  \begin{enumerate}
  \item[(a)] For every $t\in [0,T]$,
  the set $\{ \calt(v_n)(t): n\in \N\}$ is precompact in
  $\ell^2$.
  \item[(b)] The  sequence $\{\calt(v_n)\}_{n=1}^\infty$
  is equicontinuous on $[0,T]$.
  \end{enumerate}
    
    The equicontinuity of 
    $\{\calt(v_n)\}_{n=1}^\infty$
    follows from \eqref{wc_difp1}
    and the boundedness
    of $\{v_n\}_{n=1}^\infty$
    in $L^2(0,T; H)$
    due to the fact that  $v_n \to v$ weakly in 
  $L^2(0,T; H)$. It remains to show (a)
  for which we will prove 
  the set $\{ \calt(v_n)(t): n\in \N\}$
  is totally bounded in $\ell^2$.

  For every $j\in \Z$, let $e_j 
  = (\delta_{j,k})_{k\in \Z}$.
  Then $\{e_j\}_{j\in \Z}$  
  is 
  an orthonormal basis of $\ell^2$.
   Given $m\in \N$,  let $P_m:
  \ell^2 \to \ \text{span}\{e_j: |j| \le m\}$
  be the projection operator
  and $ Q_m =I-P_m$.
  Given $t\in [0,T]$, since
  $\sigma (t, \xi(t)): H \to \ell^2$ is Hilbert-Schmidt we  
  get 
  $$
 \lim_{m\to \infty}
  \| Q_m \sigma (t, \xi(t))\|_{L_2(H, \ell^2)}^2
  =0,
  $$
  which along with the dominated
  convergence theorem implies that
  for every $t\in [0,T]$,
  \be\label{wc_difp1a} 
  \lim_{m\to \infty}
   \int_0^t
   \| Q_m \sigma (s, \xi(s))\|_{L_2(H, \ell^2)}^2
   ds
  =0.
 \ee

 On the  other hand,  by    \eqref{wc_dif 1} we have
  $$ 
\| \calt (v_n) (t)\|
 \le 
\int_0^t
\| \sigma (s, \xi (s)) v_n (s) \|   ds
\le
\int_0^t
\| \sigma (s, \xi (s))\|_{L_2 (H,\ell^2)}  
\| v_n (s) \|  _H ds
$$
\be\label{wc_difp3}
\le
\left (
\int_0^t
\| \sigma (s, \xi (s))\|_{L_2 (H,\ell^2)} ^2 ds
\right )^{\frac 12} 
\left (
\int_0^t \| v_n (s) \|^2  _H ds
\right )^{\frac 12}.
\ee
Similarly, for every $m\in \N$,  we have
  $$ 
\|Q_m  \calt (v_n) (t)\|
 \le 
\int_0^t
\|Q_m \sigma (s, \xi (s)) v_n (s) \|   ds
$$
\be\label{wc_difp3a}
\le
\left (
\int_0^t
\|Q_m  \sigma (s, \xi (s))\|_{L_2 (H,\ell^2)} ^2 ds
\right )^{\frac 12} 
\left (
\int_0^t \| v_n (s) \|^2  _H ds
\right )^{\frac 12}.
\ee

  Since $\{v_n\}_{n=1}^\infty$
   is bounded  in $L^2(0,T; H)$,
   we  see  from \eqref{wc_difp3}-\eqref{wc_difp3a}
    that
   there exists  a positive number $c_1$
   independent of $n, m\in \N$ such that
 \be\label{wc_difp4}
\| \calt (v_n) (t)\|
 \le c_1,
 \quad \forall \ n\in \N,
 \ee
 and
  \be\label{wc_difp4a}
 \|Q_m  \calt (v_n) (t)\|
 \le   c_1
\left (
\int_0^t
\|Q_m  \sigma (s, \xi (s))\|_{L_2 (H,\ell^2)} ^2 ds
\right )^{\frac 12} ,
 \quad \forall \ n, m \in \N.
\ee
By \eqref{wc_difp1a} we find that
the right-hand side of \eqref{wc_difp4a}
converges to zero as $m\to \infty$,
and hence,  for every $\eta>0$,
   there exists $m_0\in \N$ such that
   for all  $n\in \N$  and $m\ge m_0$,  
      \be\label{wc_difp5} 
 \|Q_m  \calt (v_n) (t)\|
 <{\frac 14} \eta.
 \ee
 
 By \eqref{wc_difp4} we 
 see that
 $\{P_{m_0} 
   ( \calt (v_n) (t) ) \}_{n=1}^\infty$
    is bounded in a $(2m_0 +1)$-dimensional space,
    and hence it is precompact.
    Consequently,  
    $\{P_{m_0} 
   ( \calt (v_n) (t) ) \}_{n=1}^\infty$
   has a finite  open cover of radius
   ${\frac 14}\eta$,
   which along with \eqref{wc_difp5}
   shows that the sequence
       $\{ 
   \calt (v_n) (t)  \}_{n=1}^\infty$
   has a finite  open cover of radius
   $ \eta$.
   In other words, 
   the sequence $\{ 
   \calt (v_n) (t)  \}_{n=1}^\infty$
   is totally bounded and hence precompact
   in $\ell^2$.
   
   Then by (a) and (b) we infer that
   there exists a subsequence
   $\{v_{n_k}\}_{k=1}^\infty$ of
   $\{v_{n} \}_{n=1}^\infty$
   such that
     $ \calt(v_{n_k})   \to \calt (v) $  strongly
  in   $ C([0,T],   \ell^2)$.
  By  a contradiction argument, we conclude that
  the entire sequence    $ \calt(v_{n })   \to \calt (v) $  strongly
  in   $ C([0,T],   \ell^2)$.

 (ii). Let $\{v_n\}_{n=1}^\infty$ be a bounded sequence
 in $L^2(0,T; H)$. We will prove
 the sequence
  $\{ \calt (v_n)   \}_{n=1}^\infty$
  is precompact in $C([0,T], \ell^2)$.
  Since   $\{v_n\}_{n=1}^\infty$ 
  is  bounded, there exists $v\in L^2(0,T; H)$
  and a subsequence $\{v_{n_k}\}_{k=1}^\infty$ 
  such that
  $v_{n_k} \to v$ weakly in $L^2(0,T; H)$.
  Then by (i) we find that
   $  \calt (v_{n_k} )  \to \calt (v)$
   strongly in   $C([0,T], \ell^2)$,
   which completes the proof.
    \end{proof}

 \begin{lem}\label{wc_sol}
 Suppose that
 \eqref{F1}, \eqref{sigma1}-\eqref{delta1}
 and \eqref{gh} hold.
Let
  $v, v_n \in L^2(0,T; H)$ 
  for all  $n\in \N$
  and  $u_v$,  $u_{v_n}$
  be the solutions of 
  \eqref{contr1}-\eqref{contr2}
  corresponding to $v$  and $v_n$,
  respectively.
  If   $v_n \to v$ weakly in 
  $L^2(0,T; H)$,
  then
  $ u_{v_n} \to u_v $ strongly
  in  $C([0,T], \ell^2)$. 
\end{lem}

     \begin{proof}
     Suppose $v_n \to v$ weakly in 
  $L^2(0,T; H)$. Then 
  $\{v_n\}_{n=1}^\infty$ is bounded
  in  $L^2(0,T; H)$.
   Similar to \eqref{exis_solp11},  we find that
  there exists   $c_1=c_1(T)>0 $ such
  that
   \be\label{wc_solp0}
  \sup_{0\le t\le T}
  \left (
  \| u_{v_n}(t) \| + \| u_v (t) \|
  \right )
  \le c_1, \quad \forall \ n\in \N.
\ee
  By \eqref{contr1}-\eqref{contr2} we get
   $$
    {\frac d{dt}}
   ( u_{v_n}-u_v)
   =-\nu A  ( u_{v_n}-u_v)
   - (f(u_{v_n}) -f(u_v))
   $$
    \be\label{wc_solp1}
   -\gamma  ( u_{v_n}-u_v)
   +\sigma(t,u_{v_n}) v_n -\sigma (t, u_v)v.
   \ee
   By
   \eqref{f1a1}, \eqref{sigrow1}, \eqref{gh},
   \eqref{defsigma2} and 
    \eqref{wc_solp0} we infer that
     \be\label{wc_solp1a}
   \|  {\frac d{dt}}
   ( u_{v_n}(t) -u_v (t))\| 
    \le c_2 (1+ \| v_n (t)  \|_H + \| v(t)  \|_H ),
    \quad \forall\  n \in \N,
    \ee
   where $c_2=c_2(T)>0$ is a constant
    independent of $n\in \N$.
    Similar  to \eqref{exis_solp12},
    by \eqref{wc_solp1} we have 
         \be\label{wc_solp2}
    {\frac d{dt}}
    \| u_{v_n}-u_v\|^2
  \le  
    -2\gamma \| u_{v_n} -u_v \|^2
    +2\left (
    \sigma (t, u_{v_n})v_n - \sigma (t, u_v)v,
    u_{v_n} -u_v 
    \right ).
    \ee
    For the last term in \eqref{wc_solp2},
    by \eqref{defsigma2} we have
    $$
    2\left (
    \sigma (t, u_{v_n})v_n - \sigma (t, u_v)v,
    u_{v_n} -u_v 
    \right )
    $$
     $$
   =  2\left (
     ( \sigma (t, u_{v_n})-  \sigma (t, u_v) ) v_n  , 
    u_{v_n} -u_v 
    \right )
    + 
    2\left (
       \sigma (t, u_v) ( v_n  -v), 
    u_{v_n} -u_v 
    \right )
    $$
    $$
    =2\left (
    \sum_{k=1}^\infty
  \left (
    \sigma_k (u_{v_n}) - \sigma_k (u_v)
    \right )v_{n,k},\ 
     u_{v_n} -u_v 
    \right )
    $$
      \be\label{wc_solp3}
    +2\left (
    \sum_{k=1}^\infty
  \left (h_k(t) +  \sigma_k (u_v)
    \right )(v_{n,k}-v_k),\ 
     u_{v_n} -u_v 
    \right ).
    \ee
    For each $n\in \N$  and $t\in [0,T]$, set 
     \be\label{wc_solp4}
    \psi_n  (t)
    =\int_0^t
     \sum_{k=1}^\infty
  \left (h_k(s) +  \sigma_k (u_v(s))
    \right )(v_{n,k}(s) -v_k(s) ) ds 
    =\int_0^t
    \sigma (s, u_v(s)) (v_n(s) - v(s)) ds.
\ee
   Since  $v_n \to v$ weakly in 
  $L^2(0,T; H)$,
by Lemma \ref{wc_dif} we get
  \be\label{wc_solp4a}
    \psi_n 
    \to 0  \  \text{ in } \  C([0,T], \ell^2)
    \  \text{ as } \ n \to \infty,
    \ee

Note that
  $$ 
    2\left (
    \sum_{k=1}^\infty
  \left (h_k(t) +  \sigma_k (u_v)
    \right )(v_{n,k}-v_k),\ 
     u_{v_n} -u_v 
    \right )
    =2
        \left ({\frac d{dt}} \psi_n,\ 
     u_{v_n} -u_v 
    \right )
    $$
   \be\label{wc_solp5}
    =2 {\frac d{dt}}
        \left (  \psi_n (t) ,\ 
     u_{v_n}(t) -u_v (t)
    \right )-
    2
       \left (  \psi_n (t) ,\ 
     {\frac {d}{dt}} (u_{v_n}(t) -u_v (t))
    \right ).
    \ee
    By \eqref{wc_solp3}  and \eqref{wc_solp5} we obtain
       $$
    2\left (
    \sigma (t, u_{v_n})v_n - \sigma (t, u_v)v,
    u_{v_n} -u_v 
    \right )
    $$
     $$
    =2\left (
    \sum_{k=1}^\infty
  \left (
    \sigma_k (u_{v_n}) - \sigma_k (u_v)
    \right )v_{n,k},\ 
     u_{v_n} -u_v 
    \right )
    $$
       \be\label{wc_solp6}
     +2{\frac d{dt}}
        \left (  \psi_n (t) ,\ 
     u_{v_n}(t) -u_v (t)
    \right )-
    2
       \left (  \psi_n (t) ,\ 
     {\frac {d}{dt}} (u_{v_n}(t) -u_v (t))
    \right ).
    \ee
    It follows from \eqref{wc_solp2}
    and \eqref{wc_solp6} that
        $$
    {\frac d{dt}}
    \| u_{v_n}-u_v\|^2
  \le  
    -2\gamma \| u_{v_n} -u_v \|^2
   +  2\left (
    \sum_{k=1}^\infty
  \left (
    \sigma_k (u_{v_n}) - \sigma_k (u_v)
    \right )v_{n,k},\ 
     u_{v_n} -u_v 
    \right )
    $$
       \be\label{wc_solp7}
        +2{\frac d{dt}}
        \left (  \psi_n (t) ,\ 
     u_{v_n}(t) -u_v (t)
    \right )
    -
    2
       \left (  \psi_n (t) ,\ 
     {\frac {d}{dt}} (u_{v_n}(t) -u_v (t))
    \right ).
    \ee
    
    We  now deal with the right-hand side of 
    \eqref{wc_solp7}.
    For the second term on the 
    right-hand side of \eqref{wc_solp7},
    by \eqref{sigrow2}
    and  \eqref{wc_solp0}  we get
   $$
    2\left (
    \sum_{k=1}^\infty
  \left (
    \sigma_k (u_{v_n}) - \sigma_k (u_v)
    \right )v_{n,k},\ 
     u_{v_n} -u_v 
    \right )
    $$
  \be\label{wc_solp7a}
    \le 2 \left (
    \sum_{k=1}^\infty
    \| \sigma_k(u_{v_n})-\sigma_k (u_v )\|^2
    \right )^{\frac 12}\| v_n \|_H\| u_{v_n} -u_v\|
    \le 
    c_3 \|\delta\| \| v_n \|_H\| u_{v_n} -u_v\|^2,
   \ee
    where $c_3=c_3(T)>0$ is  a constant independent of $n
   \in \N$.

%
%

        For the last term on the 
    right-hand side of \eqref{wc_solp7},
    by \eqref{wc_solp1a}
 we have
     \be\label{wc_solp9}
    2 \Big | 
       \left (  \psi_n (t) ,\ 
     {\frac {d}{dt}} (u_{v_n}(t) -u_v (t))
    \right )
     \Big |
     \le
     2 c_2 (1+ \| v_n (t)  \|_H  + \| v(t)  \|_H )
     \| \psi_n (t) \|  .
     \ee
     It follows from \eqref{wc_solp7}-\eqref{wc_solp9} that
         $$
    {\frac d{dt}}
    \| u_{v_n}(t)-u_v(t)\|^2
  \le  
   ( -2\gamma 
   + c_3 \| \delta \| \|v_n (t) \|_H )
   \| u_{v_n}(t) -u_v(t) \|^2
   $$
    \be\label{wc_solp10}
   +
        2{\frac d{dt}}
        \left (  \psi_n (t) ,\ 
     u_{v_n}(t) -u_v (t)
    \right )
   +
          2 c_2 (1+ \| v_n (t)  \|_H  + \| v(t)  \|_H )
     \| \psi_n (t) \|.
     \ee
     Due to  $u_{v_n}(0)=u_v(0)=u_0$,
     by integrating  \eqref{wc_solp10}
     on $(0,t)$  we get
       $$ 
    \| u_{v_n}(t)-u_v(t)\|^2
  \le  \int_0^t
   ( -2\gamma 
   + c_3 \| \delta \| \|v_n (s) \|_H )
   \| u_{v_n}(s) -u_v(s) \|^2 ds
   $$
   $$
   +
        2 
        \left (  \psi_n (t) ,\ 
     u_{v_n}(t) -u_v (t)
    \right )
   +
          2 c_2 \int_0^t (1+ \| v_n (s)  \|_H  + \| v(s)  \|_H )
     \| \psi_n (s) \| ds,
   $$
   which shows that for all $t\in [0,T]$,
     $$ 
    \sup_{0\le r \le t}
    \| u_{v_n}(r)-u_v(r)\|^2
  \le  \int_0^t
   ( -2\gamma 
   + c_3 \| \delta \| \|v_n (s) \|_H )
   \| u_{v_n}(s) -u_v(s) \|^2 ds
   $$
    \be\label{wc_solp10a}
   +
        2    \sup_{0\le r \le t}
        \left (  \psi_n (r) ,\ 
     u_{v_n}(r) -u_v (r)
    \right )
   +
          2 c_2 \int_0^t (1+ \| v_n (s)  \|_H  + \| v(s)  \|_H )
     \| \psi_n (s) \| ds.
     \ee
     For the first term on the right-hand side
     of \eqref{wc_solp10a} we have
     $$
        \int_0^t
   ( -2\gamma 
   + c_3 \| \delta \| \|v_n (s) \|_H )
   \| u_{v_n}(s) -u_v(s) \|^2 ds
   $$
    \be\label{wc_solp10b} 
   \le
     \int_0^t
   ( -2\gamma 
   + c_3 \| \delta \| \|v_n (s) \|_H )
   \sup_{0\le r \le s}
   \| u_{v_n}(r) -u_v(r) \|^2 ds.
   \ee
      For the second  term on the right-hand side
     of \eqref{wc_solp10a} we get
     $$
      2    \sup_{0\le r \le t}
        \left (  \psi_n (r) ,\ 
     u_{v_n}(r) -u_v (r)
    \right )
    \le
    2 \| \psi_n \|_{C([0,T], \ell^2)}
       \sup_{0\le r \le t} \|
     u_{v_n}(r) -u_v (r)
    \|
     $$
      \be\label{wc_solp10c} 
      \le
      {\frac 12}
       \sup_{0\le r \le t} \|
     u_{v_n}(r) -u_v (r)
    \|^2
    + 2 \| \psi_n \|_{C([0,T], \ell^2)}^2.
    \ee
      For the last  term on the right-hand side
     of \eqref{wc_solp10a} we  have, 
     for $t\in [0,T]$,
     $$
     2 c_2 \int_0^t (1+ \| v_n (s)  \|_H  + \| v(s)  \|_H )
     \| \psi_n (s) \| ds
    $$
      \be\label{wc_solp10d} 
      \le
       2 c_2T^{\frac 12}
       \| \psi_n  \| _{C([0,T], \ell^2)}
       \left (
       T^{\frac 12} + 
         \| v_n   \|_{L^2(0,T; H)}
         +   \| v  \| _{L^2(0,T; H)} \right ) .
     \ee
     
     It follows from \eqref{wc_solp10a}-\eqref{wc_solp10d}
     that  for all $t\in [0,T]$,
     $$ 
    \sup_{0\le r \le t}
    \| u_{v_n}(r)-u_v(r)\|^2
  \le   
        2
     \int_0^t
   ( -2\gamma 
   + c_3 \| \delta \| \|v_n (s) \|_H )
   \sup_{0\le r \le s}
   \| u_{v_n}(r) -u_v(r) \|^2 ds 
     $$
       \be\label{wc_solp10e} 
      + 4 \| \psi_n \|_{C([0,T], \ell^2)}^2 
   +
       4 c_2T^{\frac 12}
       \| \psi_n  \| _{C([0,T], \ell^2)}
       \left (
       T^{\frac 12} + 
         \| v_n   \|_{L^2(0,T; H)}
         +   \| v  \| _{L^2(0,T; H)} \right ) .
     \ee
     By \eqref{wc_solp10e} and Gronwall\rq{}s lemma we obtain,
     for all 
    $t\in [0,T]$,
      $$
    \sup_{0\le r \le t}
    \| u_{v_n}(r)-u_v(r)\|^2
  \le   
  4 \| \psi_n \|_{C([0,T], \ell^2)}^2
   e^{
  \int_0^t 
   (-4\gamma 
   +2 c_3 \| \delta \| \|v_n (s) \|_H ) ds
  }
   $$
   $$
   + 
       4 c_2T^{\frac 12}
       \| \psi_n  \| _{C([0,T], \ell^2)}
       \left (
       T^{\frac 12} + 
         \| v_n   \|_{L^2(0,T; H)}
         +   \| v  \| _{L^2(0,T; H)} \right )
     e^{
  \int_0^t 
   (-4\gamma 
   +2 c_3 \| \delta \| \|v_n (s) \|_H ) ds
  },
  $$ 
  and hence
     $$
    \sup_{0\le r \le T}
    \| u_{v_n}(r)-u_v(r)\|^2
  \le   
  4 \| \psi_n \|_{C([0,T], \ell^2)}^2
   e^{  
   (-4\gamma  T
   +2 c_3 T^{\frac 12}  \| \delta \| \|v_n   \|_{L^2(0,T; H)}  
  }
   $$
      \be\label{wc_solp11} 
   + 
       4 c_2T^{\frac 12}
       \| \psi_n  \| _{C([0,T], \ell^2)}
       \left (
       T^{\frac 12} + 
         \| v_n   \|_{L^2(0,T; H)}
         +   \| v  \| _{L^2(0,T; H)} \right )
    e^{  
   (-4\gamma  T
   +2 c_3 T^{\frac 12}  \| \delta \| \|v_n   \|_{L^2(0,T; H)}  
  }.
\ee
   Since $\{v_n\}_{n=1}^\infty$  is bounded
 in $L^2(0,T; H)$, by \eqref{wc_solp11}
 we infer that there exists a positive number 
 $c_4=c_4(T)$ independent of $n\in \N$ such that
 \be\label{wc_solp12}
   \sup_{t\in [0,T]}
    \| u_{v_n}(t)-u_v(t)\|^2
  \le c_4 
  \left ( \| \psi_n \|^2_{C([0,T], \ell^2) }
  + \| \psi_n \|_{C([0,T], \ell^2) }
  \right ).
  \ee
It follows from  
  \eqref{wc_solp4a}
  and
  \eqref{wc_solp12}
   that
  $$
   \sup_{t\in [0,T]}
    \| u_{v_n}(t)-u_v(t)\|^2 
    \to 0 
    \  \text{ as } \ n \to \infty,
    $$
    which   concludes  the proof.
      \end{proof}

  We now define   $\calg^0: C([0,T], U) \to C([0,T], \ell^2)$
by, for every $\xi \in  C([0,T], U)$,
\be\label{calg0}
\calg^0 (\xi)
=
\left \{
\begin{array}{ll}
u_v & \text{ if } \xi= \int_0^\cdot v(t) dt
\ \text{ for some } v\in L^2(0,T; H);\\
0, &  \text{ otherwise} ,
\end{array}
\right.
\ee
 where $u_v$ is the
 solution
 of \eqref{contr1}-\eqref{contr2}.
 
 Given  
 $\phi \in  C([0,T], \ell^2)$, denote by
\be\label{rate_LS}
 I(\phi)
 =\inf
 \left \{
 {\frac 12} \int_0^T \| v(s)\|_H^2 ds:
 \ v\in L^2(0,T; H), \ u_v =\phi
 \right \},
\ee
 where $u_v$ is the
 solution
 of \eqref{contr1}-\eqref{contr2}.
 Again, by default, the infimum of the empty
 set is taken to be $\infty$.

 We will prove that 
 the family $\{u^\eps\}$
 satisfies the Laplace principle
 in $C([0,T],\ell^2)$
 with the  rate function
as given by \eqref{rate_LS}.
To that end, we need to show
 $\calg^\eps$ and $\calg^0$
 fulfill conditions  ({\bf  H1})
 and  ({\bf  H2})
 in terms of  Proposition \ref{LP1}.
 The following lemma confirms that 
 condition  ({\bf  H2}) is satisfied.

 \begin{lem}\label{prh1}
 Suppose that
 \eqref{F1}, \eqref{sigma1}-\eqref{delta1}
 and \eqref{gh} hold.
Then for every $N<\infty$, the set
\be\label{prh1 1}
K_N
=\left \{
\calg^0
\left (
\int_0^\cdot v(t) dt
\right ) : \  v\in S_N
\right \}
\ee
is a compact subset
of $C([0,T], \ell^2)$,
where $S_N$ is the set as defined by
\eqref{pre0a}.
 \end{lem}

 \begin{proof}
 By \eqref{calg0}  and \eqref{prh1 1} we see that
 $$ K_N
=\left \{
u_v : \  v\in S_N
\right \}
=
\left \{
u_v : \  v\in L^2(0,T; H), \ \int_0^T \| v(t)\|_H^2 dt
\le N
\right \},
$$
where $u_v$
is the solution of \eqref{contr1}-\eqref{contr2}.

Let $\{u_{v_n}\}_{n=1}^\infty$
 be a sequence
in $K_N$.
Then $ v_n \in L^2(0,T; H)$
and 
$\int_0^T \| v_n(t)\|_H^2 dt
\le N$,
which shows that
there exists $v\in S_N$ and a subsequence
$\{ {v_{n_k}}\}_{k=1}^\infty$
such that
$v_{n_k} \to v$ weakly in 
$L^2(0,T; H)$.
Then by Lemma \ref{wc_sol} we find that
$u_{v_{n_k}} \to u_v$ strongly 
in $C([0,T], \ell^2)$, as desired.
 \end{proof}

In order to  prove   
    ({\bf  H1}), we need the
    following  property of the measurable
    map 
   $\calg^\eps$.

 \begin{lem}\label{gep}
 Suppose that
 \eqref{F1}, \eqref{sigma1}-\eqref{delta1}
 and \eqref{gh} hold, and
 $v\in \cala_N$
for some  $N<\infty$.
If $u^\eps_v
=
 \calg^\eps
\left (
W +\eps^{-\frac 12}\int_0^\cdot
v (t) dt
\right )$, then
$u^\eps_v$ is 
the unique solution to  
 \be\label{gep 1}
 d   u_{v}^\eps  
 + \left (  \nu Au^\eps_{v}     
 +f(u^\eps_{v}  )  
+ \gamma  u^\eps_{v}  
\right )  dt  =
\left (
g (t)  
+\sigma(t, u^\eps_{v}  )  v
\right )  dt
 + \sqrt{\eps} \sigma(t, u^\eps_{v}  )  
   dW   ,
\ee
 with  initial condition 
 $
 u^\eps_{v} (0)=u_0\in\ell^2. $

  Furthermore,
     for each $R>0$  
    there  exists
       $C_2 =C_2  (R, T,N)>0$ such that  
       for any $u_{0} 
       \in \ell^2$ with 
    $\| u_{0} \|\le R  $   and any
    $v \in \cala_N$,
    the solution $u_{v}^\eps$    
    satisfies for all $\eps\in (0,1)$,
  \be\label{gep 2}
 \E\left (
 \| u^\eps_v \|^2_{C([0,T], \ell^2)}
 \right )
 \le C_2.
 \ee
 \end{lem}

 \begin{proof}
 Given $\eps>0$,
since $v\in \cala_N$,
by Girsanov\rq{}s theorem
we know that
$\widetilde{W}
=
W +\eps^{-\frac 12}\int_0^\cdot
v (t) dt$
is a 
cylindrical Wiener
process with identity  covariance operator
under the probability
$\widetilde{P}^\eps_{v}$ as given by
$$
 {\frac {d \widetilde{P}^\eps_{v}}{dP}}
 =
 \exp
 \left \{
 -\eps^{-\frac 12}
 \int_0^T v
 (t) dW
 -{\frac 12}\eps^{-1}
 \int_0^T \| v(t) \|_H^2 dt
 \right \} ,
 $$
 which implies that
 $u_{v}^\eps
 =\calg^\eps
 (\widetilde{W})
 =
 \calg^\eps
\left (
W +\eps^{-\frac 12}\int_0^\cdot
v (t) dt
\right )$
is the unique solution of 
\eqref{intr5}-\eqref{intr6}
with $W$ replaced by
$\widetilde{W}$.
In other words,
$u^\eps_{v}$ is the unique solution of
\eqref{gep 1} with initial condition
$
 u^\eps_{v} (0)=u_0$. 
 It remains to show  \eqref{gep 2}.

 By \eqref{gep 1} and It\^{o}\rq{}s formula, we
 have for all $t\in [0,T]$,  $P$-almost surely,
   $$ 
   \|  u^\eps_{v} (t) \|^2
   + 2 \nu \int_0^t  \| B u^\eps_{v} (s)\|^2ds
    + 2\int_0^t  (f( u^\eps_{v} (s)),  u^\eps_{v} (s))ds
   + 2 \gamma\
   \int_0^t  \|  u^\eps_{v} (s) \|^2 ds
   $$
   $$
   = \| u_0\|^2 +  2 
   \int_0^t (g(s),  u^\eps_{v} (s)) ds
   +2
   \int_0^t
    (\sigma (s,  u^\eps_{v} (s)) v(s),  u^\eps_{v}(s))
    ds
   $$
\be\label{gep p1}
   +\eps \int_0^t
   \| \sigma (s, u^\eps_v(s) )\|^2_{L_2(H,
   \ell^2)} ds
   +2\sqrt{\eps}
   \int_0^t
   (u^\eps_v (s), \sigma (s, u^\eps_v (s)) dW ).
\ee
   By   \eqref{f1} and \eqref{f2}
  we know
 \be\label{gep p2}
  2 \nu \int_0^t  \| B u^\eps_{v} (s)\|^2ds
    + 2\int_0^t  (f( u^\eps_{v} (s)),  u^\eps_{v} (s))ds
    \ge 0.
    \ee
   We also  have
 \be\label{gep p3} 
  2\int_0^t  |(g(s), u^\eps_{v}(s))| ds
  \le \int_0^ t  \|  u^\eps_{v} (s) \|^2 
  ds +  \int_0^t \| g(s) \|^2 ds .
  \ee
By \eqref{defsigma3}
  and \eqref{sigrow1} we  obtain
 $$
  2\int_0^t | (\sigma (s, u^\eps_v (s)) v(s), u^\eps_v (s))|
  ds
  \le 
    2\int_0^t 
  \| \sigma (s, u^\eps_v (s) )\|_{L(H,\ell^2)}
   \| u^\eps_v (s) \|\|v(s) \|_H ds
  $$
  $$
  \le
    2
    \left (
    \int_0^t 
  \| \sigma (s, u^\eps_v (s) )\|_{L(H,\ell^2)}^2
   \| u^\eps_v (s) \|^2  ds
   \right )^{\frac 12}
   \left (
   \int_0^t \| v(s) \|^2_H ds 
 \right )^{\frac 12}
 $$
   $$
  \le
    2N^{\frac 12}
    \sup_{0\le s \le t}
     \| u^\eps_v (s) \|
    \left (
    \int_0^t 
  \| \sigma (s, u^\eps_v (s) )\|_{L(H,\ell^2)}^2
   ds
   \right )^{\frac 12}
  $$
  $$
  \le
  {\frac 14}     
    \sup_{0\le s \le t}
     \| u^\eps_v (s) \|^2
     +
  4N  
    \int_0^t 
  \| \sigma (s, u^\eps_v (s) )\|_{L(H,\ell^2)}^2
   ds 
  $$
  $$
  \le
  {\frac 14}     
    \sup_{0\le s \le t}
     \| u^\eps_v (s) \|^2
     +
  4N    \sum_{k=1}^\infty
    \int_0^t 
      \| h_k(s) + \sigma_k (u^\eps_v (s)) \|^2 
   ds 
  $$
   $$
  \le
  {\frac 14}     
    \sup_{0\le s \le t}
     \| u^\eps_v (s) \|^2
     +
  8N T   \sum_{k=1}^\infty
      \| h_k\|_{L^\infty(0,T; \ell^2)}^2
         +
  8N    \sum_{k=1}^\infty
    \int_0^t 
      \|   \sigma_k (u^\eps_v (s)) \|^2 
   ds 
   $$
   $$
  \le
  {\frac 14}     
    \sup_{0\le s \le t}
     \| u^\eps_v (s) \|^2
     +
  8N T   \sum_{k=1}^\infty
      \| h_k\|_{L^\infty(0,T; \ell^2)}^2
      $$
    \be\label{gep p4}
         +
  16 N  \alpha^2 \| \delta\|^2 T
  +  16 N  \alpha^2 \| \delta\|^2 
    \int_0^t 
      \|   u^\eps_v (s ) \|^2 
   ds .
   \ee
 Similarly,
by \eqref{defsigma3}
  and \eqref{sigrow1} we have
  for all $\eps\in (0,1)$,
    \be\label{gep p5}
     \eps \int_0^t
   \| \sigma (s, u^\eps_v(s) )\|^2_{L_2(H,
   \ell^2)} ds
  \le
  2 T   \sum_{k=1}^\infty
      \| h_k\|_{L^\infty(0,T; \ell^2)}^2
         +
 4  \alpha^2 \| \delta\|^2 T
  +  4 \alpha^2 \| \delta\|^2 
    \int_0^t 
      \|   u^\eps_v (s ) \|^2 
   ds .
   \ee
   It follows from
   \eqref{gep p1}-\eqref{gep p5} that
   for all $t\in [0,T]$,  $P$-almost surely,
   $$ 
   \|  u^\eps_{v} (t) \|^2
   \le
    \| u_0\|^2
    +
    \left (
    1-2\gamma +
    4\alpha^2 \| \delta \|^2
    (1+ 4N)
    \right ) 
   \int_0^t  \|  u^\eps_{v} (s) \|^2 ds
   $$
     $$
+
  {\frac 14}     
    \sup_{0\le s \le t}
     \| u^\eps_v (s) \|^2
     +
     4\alpha^2 \| \delta \|^2 T
     (1+4N)
     +2T( 1+4 N)
    \sum_{k=1}^\infty
      \| h_k\|_{L^\infty(0,T; \ell^2)}^2
        $$
    $$
           +\int_0^t \| g(s) \|^2 ds
   +2\sqrt{\eps}
   \int_0^t
   (u^\eps_v (s), \sigma (s, u^\eps_v (s)) dW ),
$$
   which implies that
    for all $t\in [0,T]$,   
   $$ 
  {\frac 34}
  \E \left (
   \sup_{0\le s\le t}
   \|  u^\eps_{v} (s) \|^2
   \right )
   \le
    \| u_0\|^2
    +
    \left (
    1-2\gamma +
    4\alpha^2 \| \delta \|^2
    (1+ 4N)
    \right ) \E \left (
   \int_0^t  \|  u^\eps_{v} (s) \|^2 ds
   \right )
   $$
     $$ 
     +
     4\alpha^2 \| \delta \|^2 T
     (1+4N)
     +2T( 1+4 N)
    \sum_{k=1}^\infty
      \| h_k\|_{L^\infty(0,T; \ell^2)}^2
        $$
     \be\label{gep p6} 
           +\int_0^t \| g(s) \|^2 ds
   +
   \E \left (
   \sup_{0\le r\le t}
  \left | \int_0^r  2\sqrt{\eps}
   (u^\eps_v (s), \sigma (s, u^\eps_v (s)) dW (s))
   \right |
   \right ).
\ee
     For the last term in
   \eqref{gep p6}, by the 
   Burkholder inequality we get
   for $\eps\in (0,1)$,
 $$
   \E\left (
   \sup_{0\le r\le t}
  \left | \int_0^r  2\sqrt{\eps}
   (u^\eps_v (s), \sigma (s, u^\eps_v (s)) dW (s))
   \right |
   \right )
   $$
   $$
   \le
   6
   \E \left (
   \left (
   \int_0^t 
  \|  u^\eps_v (s) \|^2 
\|   \sigma (s, u^\eps_v (s))\|_{L_2(H,\ell^2)}^2
ds
   \right )^{\frac 12}
   \right )
   $$
     $$
   \le
   {\frac 14}
   \E
   \left (
   \sup_{0\le s\le t} \| u^\eps_v (s)\|^2
   \right )
   +
   36
   \E \left ( 
   \int_0^t   
\|   \sigma (s, u^\eps_v (s))\|_{L_2(H,\ell^2)}^2
ds 
   \right ),
   $$
  which along with \eqref{gep p5}
  shows that 
  $$
   \E\left (
   \sup_{0\le r\le t}
  \left | \int_0^r  2\sqrt{\eps}
   (u^\eps_v (s), \sigma (s, u^\eps_v (s)) dW (s))
   \right |
   \right )
   \le
    {\frac 14}
   \E
   \left (
   \sup_{0\le s\le t} \| u^\eps_v (s)\|^2
   \right )
   $$
    \be\label{gep p7}
    +
  72T   \sum_{k=1}^\infty
      \| h_k\|_{L^\infty(0,T; \ell^2)}^2
         +
 144 \alpha^2 \| \delta\|^2 T
  + 14 4 \alpha^2 \| \delta\|^2 
    \E \left (
    \int_0^t 
      \|   u^\eps_v (s ) \|^2 
   ds
   \right )  .
   \ee
  By \eqref{gep p6} and \eqref{gep p7} we get
    for all $t\in [0,T]$,   
   $$  
  \E \left (
   \sup_{0\le s\le t}
   \|  u^\eps_{v} (s) \|^2
   \right )
   \le
    2\| u_0\|^2
    +2
    \left (
    1-2\gamma +
    4\alpha^2 \| \delta \|^2
    (37+ 4N)
    \right ) 
   \int_0^t \E \left (
   \sup_{0\le r \le s } \|  u^\eps_{v} (r) \|^2 
   \right ) dr
   $$
       \be\label{gep p8} 
     +
     8\alpha^2 \| \delta \|^2 T
     (37+4N)
     +4T( 37+4 N)
    \sum_{k=1}^\infty
      \| h_k\|_{L^\infty(0,T; \ell^2)}^2
  +2\int_0^T \| g(s) \|^2 ds.
  \ee
  By \eqref{gep p8} and
  Gronwall\rq{}s lemma we obtain
  for all $t\in [0,T]$,
        \be\label{gep p9} 
  \E \left (
   \sup_{0\le s\le t}
   \|  u^\eps_{v} (s) \|^2
   \right )
   \le
   c_1 e^{c_2 t},
\ee
   where
   $$
   c_1=
     2\| u_0\|^2
       +
     8\alpha^2 \| \delta \|^2 T
     (37+4N)
     +4T( 37+4 N)
    \sum_{k=1}^\infty
      \| h_k\|_{L^\infty(0,T; \ell^2)}^2
  +2\int_0^T \| g(s) \|^2 ds,
  $$
  and
  $$ c_2
  =
    2
    \left (
    1-2\gamma +
    4\alpha^2 \| \delta \|^2
    (37+ 4N)
    \right ) .
    $$
  Then \eqref{gep 2} follows from \eqref{gep p9}
  for $t=T$.
 \end{proof}

 We now prove $\calg^\eps$
 and $\calg^0$ satisfy
 ({\bf {H1}}).

 \begin{lem}\label{prh2}
 Suppose that
 \eqref{F1}, \eqref{sigma1}-\eqref{delta1}
 and \eqref{gh} hold, and
 $\{v^\eps\}  \subseteq \cala_N$
for some  $N<\infty$.
If $\{v^\eps\} $
converges in distribution
to $v$ as $S_N$-valued random variables,
then
$ 
\calg^\eps
\left (
W +\eps^{-\frac 12}\int_0^\cdot
v^\eps (t) dt
\right )$
converges
to $\calg^0\left (
\int_0^\cdot
v(t) dt
\right )$
in $ C([0,T], \ell^2)$
in distribution.
\end{lem}

\begin{proof}
Let 
 $u_{v^\eps}^\eps 
 =
 \calg^\eps
\left (
W +\eps^{-\frac 12}\int_0^\cdot
v^\eps (t) dt
\right )$. Then by Lemma 
\ref{gep} we see that
$u_{v^\eps}^\eps $  
is the   solution to the equation: 
 \be\label{prh2 p1}
 d   u_{v^\eps}^\eps  
 + \left (  \nu Au^\eps_{v^\eps}     
 +f(u^\eps_{v^\eps}  )  
+ \gamma  u^\eps_{v^\eps}  
\right )  dt  =
\left (
g (t)  
+\sigma(t, u^\eps_{v^\eps}  )  v^\eps
\right )  dt
 + \sqrt{\eps} \sigma(t, u^\eps_{v^\eps}  )  
   dW   ,
\ee
 with  initial data
 \be\label{prh2 p2}
 u^\eps_{v^\eps} (0)=u_0\in\ell^2.
\ee
Let $u_v=
  \calg^0\left (
\int_0^\cdot
v(t) dt
\right )   $. Then
  $u_v$
is the solution 
 to \eqref{contr1}-\eqref{contr2}.
So we only need to show
$u^\eps_{v^\eps} $ converges to $u_v$
in $C([0,T], \ell^2)$
in distribution as $\eps \to 0$.
To that end, we first  establish the convergence
of 
$u^\eps_{v^\eps} - u_{v^\eps}$
with
$u_{v^\eps}=  \calg^0\left (
\int_0^\cdot
v^\eps (t) dt \right ) $
as in \cite{brz1}.

By  
\eqref{contr1}-\eqref{contr2} we have 
    \be\label{prh2 p3}
{\frac {d   u_{v^\eps} }{dt}}
=-  \nu Au_ {v^\eps}    
-f(u_ {v^\eps} )  
- \gamma  u_ {v^\eps}    
+
g (t)  
 +   \sigma(t, u_ {v^\eps} )  
    {v^\eps} ,
\ee
 with  initial data
  \be\label{prh2 p4}
 u_{v^\eps}  (0)=u_0 \in\ell^2.
 \ee
  By \eqref{prh2 p1}-\eqref{prh2 p4} we have
   $$
 d  ( u_{v^\eps}^\eps -  u_{v^\eps}  )
 +   \nu A ( u^\eps_{v^\eps}  -  u_{v^\eps  }
 ) dt 
 +
   ( f(u^\eps_{v^\eps}  ) -  f(u_{v^\eps}  ) ) dt
+ \gamma (  u^\eps_{v^\eps}  
-  u_{v^\eps } ) dt
$$
  \be\label{prh2 p5}
  =
\left ( 
 \sigma(t, u^\eps_{v^\eps}  )  v^\eps
 -
  \sigma(t, u_{v^\eps}  )  v^\eps
\right )  dt
 + \sqrt{\eps} \sigma(t, u^\eps_{v^\eps}  )  
   dW  ,
\ee
with  
$u_{v^\eps}^\eps(0) -  u_{v^\eps} (0)=0$.
By \eqref{prh2 p5} and It\^{o}\rq{}s formula we obtain
$$
 \|  u_{v^\eps}^\eps(t) -  u_{v^\eps}(t) \|^2
+2 \nu
\int_0^t
\| B ( u^\eps_{v^\eps}(s)  -  u_{v^\eps} (s)
 ) \|^2  ds
 + 2
 \int_0^t
  ( f(u^\eps_{v^\eps}  ) -  f(u_{v^\eps}),
  \ u^\eps_{v^\eps}  -  u_{v^\eps}  ) ds
  $$
  $$
  + 2  \gamma
  \int_0^t
  \| u^\eps_{v^\eps}(s)  -  u_{v^\eps}(s)
  \|^2  ds
  =
  2 
  \int_0^t
  \left ( 
 \sigma(s, u^\eps_{v^\eps} (s) )  v^\eps(s)
 -
  \sigma(s, u_{v^\eps} (s)  )  v^\eps (s),
  \  u^\eps_{v^\eps}(s)  -  u_{v^\eps}(s)
\right )  dt
$$
   \be\label{prh2 p6}
+ \eps
\int_0^t
 \|   \sigma(s, u^\eps_{v^\eps} (s)  )\|
_{L_2(H, \ell^2)}^2  ds
+ 
2\sqrt{\eps} \int_0^t
\left (
u^\eps_{v^\eps}(s)  -  u_{v^\eps}(s),\
 \sigma(s, u^\eps_{v^\eps} (s)  ) 
   dW  \right ) .
\ee
By \eqref{f2}  and
\eqref{prh2 p6} we obtain
$$
 \|  u_{v^\eps}^\eps(t) -  u_{v^\eps}(t) \|^2
   + 2  \gamma
  \int_0^t
  \| u^\eps_{v^\eps}(s)  -  u_{v^\eps}(s)
  \|^2  ds
  $$
  $$
  \le 
  2 
  \int_0^t
  \left ( 
 \sigma(s, u^\eps_{v^\eps} (s) )  v^\eps(s)
 -
  \sigma(s, u_{v^\eps} (s)  )  v^\eps (s),
  \  u^\eps_{v^\eps}(s)  -  u_{v^\eps}(s)
\right )  dt
$$
   \be\label{prh2 p6a}
+ \eps
\int_0^t
 \|   \sigma(s, u^\eps_{v^\eps} (s)  )\|
_{L_2(H, \ell^2)}^2  ds
+ 
2\sqrt{\eps} \int_0^t
\left (
u^\eps_{v^\eps}(s)  -  u_{v^\eps}(s),\
 \sigma(s, u^\eps_{v^\eps} (s)  ) 
   dW  \right ) .
\ee
 For a fixed $M>0$, define a stopping time
 \be\label{prh2 p7}
 \tau^\eps 
 =\inf \left \{
 t\ge 0: \| u^\eps_{v^\eps} (t) \| \ge M
 \right \}\wedge T.
 \ee
 As usual,  the infimum of the empty set is taken
 to be $\infty$.
 Since $\gamma \le 0$,
 by  \eqref{prh2 p6a}-\eqref{prh2 p7} we get,
 for all $t\in [0,T]$,
 $$ 
 \sup_{0\le r \le t}
 \|  u_{v^\eps}^\eps(r\wedge \tau^\eps)
  -  u_{v^\eps}( r\wedge \tau^\eps) \|^2
 \le - 2  \gamma
  \int_0^{t \wedge \tau^\eps}
  \| u^\eps_{v^\eps}(s)  -  u_{v^\eps}(s)
  \|^2  ds
  $$
  $$
  + 
  2  
  \int_0^{ t\wedge \tau^\eps}
  |\left ( 
 \sigma(s, u^\eps_{v^\eps} (s) )  v^\eps(s)
 -
  \sigma(s, u_{v^\eps} (s)  )  v^\eps (s),
  \  u^\eps_{v^\eps}(s)  -  u_{v^\eps}(s)
\right )|   dt
$$
   \be\label{prh2 p8}
+ \eps
\int_0^ { t\wedge \tau^\eps}
 \|   \sigma(s, u^\eps_{v^\eps} (s)  )\|
_{L_2(H, \ell^2)}^2  ds
+ 
2\sqrt{\eps}
\sup_{0\le r\le t}
 \left |\int_0^ { r\wedge \tau^\eps}
\left (
u^\eps_{v^\eps}(s)  -  u_{v^\eps}(s),\
 \sigma(s, u^\eps_{v^\eps} (s)  ) 
   dW  \right )
   \right | .
\ee
  
  For the first term on the right-hand side of
  \eqref{prh2 p8} we have
   \be\label{prh2 p9}
    - 2  \gamma 
  \int_0^{t \wedge \tau^\eps}
  \| u^\eps_{v^\eps}(s)  -  u_{v^\eps}(s)
  \|^2  ds
\le
 - 2  \gamma
  \int_0^{t  }
  \sup_{0\le r\le s}
  \| u^\eps_{v^\eps}({r\wedge \tau^\eps} ) 
   -  u_{v^\eps}( {r\wedge \tau^\eps})
  \|^2  ds.
  \ee
  To deal with the 
second term on the right-hand side of
  \eqref{prh2 p8},  we find from   
   \eqref{exis_sol 2}  that
  there exists a positive number
  $c_1=c_1 (N)$ such that
  $P$-almost surely,
  \be\label{prh2 p10}
  \sup_{\eps\in (0,1]}
  \ \sup_{t\in [0,T]}
  \|u_{v^\eps} (t) \|\le c_1.
  \ee
  Then by
  \eqref{sigrow2},
  \eqref{defsigma2},
    \eqref{prh2 p7}
  and \eqref{prh2 p10} we 
  infer that
  there exists
  $c_2=c_2 (N, M)>0$ such that
   $$  
  2  
  \int_0^{ t\wedge \tau^\eps}
  |\left ( 
 \sigma(s, u^\eps_{v^\eps} (s) )  v^\eps(s)
 -
  \sigma(s, u_{v^\eps} (s)  )  v^\eps (s),
  \  u^\eps_{v^\eps}(s)  -  u_{v^\eps}(s)
\right )|   ds 
$$
  $$  
 \le
  2 
  \int_0^{ t\wedge \tau^\eps}
\|
 \sigma(s, u^\eps_{v^\eps} (s) )  v^\eps(s)
 -
  \sigma(s, u_{v^\eps} (s)  )  v^\eps (s)
  \|
  \|  u^\eps_{v^\eps}(s)  -  u_{v^\eps}(s)
\|   ds 
$$ 
 $$  
 \le
    c_2 \| \delta \|  
  \int_0^{ t\wedge \tau^\eps}
\|   v^\eps(s)\|_H 
  \|  u^\eps_{v^\eps}(s)  -  u_{v^\eps}(s)
\|^2    ds 
$$
   \be\label{prh2 p11}
 \le
    c_2 \| \delta \|  
  \int_0^{ t }
\|   v^\eps(s)\|_H 
 \sup_{0\le r\le s}
  \|  u^\eps_{v^\eps}(r \wedge \tau^\eps) 
   -  u_{v^\eps}( r \wedge \tau^\eps)
\|^2    ds .
\ee

It follows from
\eqref{prh2 p8},
\eqref{prh2 p9}
and
\eqref{prh2 p11} that
for all $t\in [0,T]$,
$P$-almost surely,
 $$ 
 \sup_{0\le r \le t}
 \|  u_{v^\eps}^\eps(r\wedge \tau^\eps)
  -  u_{v^\eps}( r\wedge \tau^\eps) \|^2 
  $$
  $$
  \le 
  \int_0^{ t }
 \left (  c_2 \| \delta \|  
 \|   v^\eps(s)\|_H  -2\gamma
 \right )
 \sup_{0\le r\le s}
  \|  u^\eps_{v^\eps}(r \wedge \tau^\eps) 
   -  u_{v^\eps}( r \wedge \tau^\eps)
\|^2    ds 
  $$
    \be\label{prh2 p11a}
+ \eps
\int_0^ { T\wedge \tau^\eps}
 \|   \sigma(s, u^\eps_{v^\eps} (s)  )\|
_{L_2(H, \ell^2)}^2  ds
+ 
2\sqrt{\eps}
\sup_{0\le r\le T}
 \left |\int_0^ { r\wedge \tau^\eps}
\left (
u^\eps_{v^\eps}(s)  -  u_{v^\eps}(s),\
 \sigma(s, u^\eps_{v^\eps} (s)  ) 
   dW  \right )
   \right | .
\ee
  By \eqref{prh2 p11a} and Gronwall\rq{}s lemma,
  we get
  for all $t\in [0,T]$,
$P$-almost surely,
$$ 
 \sup_{0\le r \le t}
 \|  u_{v^\eps}^\eps(r\wedge \tau^\eps)
  -  u_{v^\eps}( r\wedge \tau^\eps) \|^2 
  $$
  \be\label{prh2 p11b}
\le  \eps c_3
\int_0^ { T\wedge \tau^\eps}
 \|   \sigma(s, u^\eps_{v^\eps} (s)  )\|
_{L_2(H, \ell^2)}^2  ds
+ 
2\sqrt{\eps}c_3
\sup_{0\le r\le T}
 \left |\int_0^ { r\wedge \tau^\eps}
\left (
u^\eps_{v^\eps}(s)  -  u_{v^\eps}(s),\
 \sigma(s, u^\eps_{v^\eps} (s)  ) 
   dW  \right )
   \right | ,
\ee
where 
$ 
c_3= 
e^{c_2 \| \delta \| T^{\frac 12} N^{\frac 12}  
  -2\gamma T
} $.

For the first term on the right-hand side of 
  \eqref{prh2 p11b},   
  by \eqref{sigrow1}
  and \eqref{defsigma3} we get 
  $$
   \eps
\int_0^ { T\wedge \tau^\eps}
 \|   \sigma(s, u^\eps_{v^\eps} (s)  )\|
_{L_2(H, \ell^2)}^2  ds
\le
2\eps
\sum_{k=1}^\infty
\int_0^ { T\wedge \tau^\eps}
\left (
\|h_k(s) \|^2
+\| \sigma_k (u^\eps_{v^\eps} (s)  )\|^2
\right ) ds
$$
$$
\le 2\eps
T
\sum_{k=1}^\infty
\| h_k\|^2_{L^\infty(0,T; \ell^2)}
+
4\eps \alpha^2
\|\delta\|^2 T
+
  4\eps \alpha^2
\|\delta\|^2 
    \int_0^ { T \wedge \tau^\eps}
    \| u^\eps_{v^\eps} (s)\|^2 ds
    $$
   \be\label{prh2 p12}
\le 2\eps
T
\sum_{k=1}^\infty
\| h_k\|^2_{L^\infty(0,T; \ell^2)}
+
4\eps \alpha^2
\|\delta\|^2 T
+
  4\eps \alpha^2
\|\delta\|^2 M^2T. 
  \ee
  By \eqref{prh2 p12} we see that
     \be\label{prh2 p13}
     \lim_{\eps \to 0}
       \eps c_3 
\int_0^ { T\wedge \tau^\eps}
 \|   \sigma(s, u^\eps_{v^\eps} (s)  )\|
_{L_2(H, \ell^2)}^2  ds
=0,\quad \text{P-almost surely}.
\ee
For the last term on the right-hand side of
\eqref{prh2 p11b},
by \eqref{prh2 p10},
\eqref{prh2 p12}
and
Doob\rq{}s maximal inequality we  obtain
$$
\E\left (
 \sup_{0\le r\le T}
 \left |\int_0^ { r\wedge \tau^\eps}
2\sqrt{\eps} \left (
u^\eps_{v^\eps}(s)  -  u_{v^\eps}(s),\
 \sigma(s, u^\eps_{v^\eps} (s)  ) 
   dW  \right )
   \right |^2
   \right )
$$
$$
\le 16 \eps
\E\left ( 
  \int_0^ { T\wedge \tau^\eps}
 \|
u^\eps_{v^\eps}(s)  -  u_{v^\eps}(s)\|^2
 \| \sigma(s, u^\eps_{v^\eps} (s)   )\|^2_{L_2(H,\ell^2)}
 ds 
   \right )
$$
$$
\le 16 \eps
(M+c_1)^2
\E\left ( 
  \int_0^ { T\wedge \tau^\eps}
 \| \sigma(s, u^\eps_{v^\eps} (s)   )\|^2_{L_2(H,\ell^2)}
 ds 
   \right )
$$
 $$
\le 32\eps 
(M+c_1)^2 
 \left (
T
\sum_{k=1}^\infty
\| h_k\|^2_{L^\infty(0,T; \ell^2)}
+
2 \alpha^2
\|\delta\|^2 T
+
 2 \alpha^2
\|\delta\|^2 M^2T
\right ),
$$
and hence
  \be\label{prh2 p14}
\lim_{\eps \to 0}
\E\left ( 
 \sup_{0\le r\le T}
 \left |
 2\sqrt{\eps} c_3
 \int_0^ { r\wedge \tau^\eps}
 \left (
u^\eps_{v^\eps}(s)  -  u_{v^\eps}(s),\
 \sigma(s, u^\eps_{v^\eps} (s)  ) 
   dW  \right )
   \right |^2
   \right ) =0.
\ee

By \eqref{prh2 p11b},
\eqref{prh2 p13} and
\eqref{prh2 p14}, we get
  \be\label{prh2 p15}
 \lim_{\eps \to 0}
 \sup_{0\le t \le T}
 \|  u_{v^\eps}^\eps(t\wedge \tau^\eps)
  -  u_{v^\eps}( t\wedge \tau^\eps) \|^2 
  =0
  \quad \text{in probability}.
\ee

 On the other hand, by
 \eqref{prh2 p7}
 and \eqref{gep 2}  we have
 for all $\eps \in (0,1)$,
   \be\label{prh2 p16}
 P(\tau^\eps <T)
 =
 P \left (
 \sup_{t\in [0,T]}
 \| u^\eps_{v^\eps} (t) \| >M 
 \right )
 \le
 {\frac 1{M^2}}
 \E \left (
 \sup_{t\in [0,T]}
 \| u^\eps_{v^\eps} (t) \|^2 
 \right )
  \le
 {\frac {C_2}{M^2}},
\ee
 where $C_2=C_2(T, N)>0$.
 It follows  from
 \eqref{prh2 p16}
 that
 for every $\eta>0$,
 $$
 P 
\left (
\sup_{0\le t \le T}
 \|  u_{v^\eps}^\eps(t )
  -  u_{v^\eps}( t ) \| >\eta
\right )
$$
$$
 \le P 
\left (
\sup_{0\le t \le T}
 \|  u_{v^\eps}^\eps(t )
  -  u_{v^\eps}( t ) \| >\eta,
  \  \tau^\eps =T
\right )
+
  P 
\left (
\sup_{0\le t \le T}
 \|  u_{v^\eps}^\eps(t )
  -  u_{v^\eps}( t ) \| >\eta,
  \  \tau^\eps <T
\right )
$$
$$
 \le P 
\left (
\sup_{0\le t \le T}
 \|  u_{v^\eps}^\eps(t \wedge  \tau^\eps  )
  -  u_{v^\eps}( t \wedge  \tau^\eps) \| >\eta  
\right )
+
   {\frac {C_2}{M^2}}.
$$
First taking the limit as $\eps \to 0$, and then
as 
$M\to \infty$, we get
from  \eqref{prh2 p15} that 
 \be\label{prh2 p20}
 \lim_{\eps \to 0} 
 ( u_{v^\eps}^\eps
  -   u_{v^\eps}  ) =0 
  \ \text{ in  } {C([0,T],\ell^2)}
\  \text{ in probability}.
\ee

Since $\{v^\eps\}$ converges
in distribution to $v$ as 
$S_N$-valued random  elements,
by Skorokhod\rq{}s theorem,
there exists a probability
space $(\widetilde{\Omega},
\widetilde{\calf}, \widetilde{P})$
and 
$S_N$-valued random variables
$\widetilde{v}^\eps$
and $\widetilde{v}$ on
$(\widetilde{\Omega},
\widetilde{\calf}, \widetilde{P})$
such that
$  \widetilde{v}^\eps$ and $
 \widetilde{v} $
 have the same distribution laws as
 $  {v^\eps}$ and $
  {v} $ respectively,  and
  $ \{\widetilde{v}^\eps\}$
 converges to
 $  \widetilde{v} $ 
  almost surely
  in $S_N$ which is equipped with the
  weak topology.
  Then by
  Lemma \ref{wc_sol} we infer that
  $$
  u_{\widetilde{v} ^\eps}
  \to u_{\widetilde{v}}
  \ \text{in } C([0,T], \ell^2)
  \ \text{almost surely},
 $$
 and hence
\be\label{prh2 p21}
  u_{\widetilde{v} ^\eps}
  \to u_{\widetilde{v}}
  \ \text{in } C([0,T], \ell^2)
  \ \text{in distribution}.
\ee
By \eqref{prh2 p21} we know 
\be\label{prh2 p22}
  u_{ {v} ^\eps}
  \to u_{ {v}}
  \ \text{in } C([0,T], \ell^2)
  \ \text{in distribution}.
\ee
By \eqref{prh2 p20}
and
\eqref{prh2 p22} we immediately get 
$$
  u^\eps_{ {v} ^\eps}
  \to u_{ {v}}
  \ \text{ in } C([0,T], \ell^2)
  \ \text{in distribution},
$$
which concludes the proof.  
\end{proof}

 As  an immediate  consequence of
 Proposition \ref{LP1}
 and Lemmas \ref{prh1} and  \ref{prh2}, we 
 finally obtain 
 Theorem \ref{main_resu}, 
 the main result of this paper.

%
%
%

\end{document}